\documentclass[11pt]{amsart}
\usepackage{amsmath,amsfonts,amssymb,amsthm,aliascnt}
\usepackage{latexsym,bm,graphicx}
\usepackage{mathrsfs}
\usepackage{color}
\usepackage[linktocpage, colorlinks,  citecolor = green,backref = page, linkcolor = blue]{hyperref}
\usepackage[all]{xy}

\usepackage{enumitem}

\newcommand{\customitemize}[1]{%
	\begin{enumerate}[label={(#1\arabic*)}, ref=#1\arabic*]
	}

\title[Fundamental groups of $4$-manifolds]{Nonnegative Ricci Curvature, Euclidean Volume Growth, and the Fundamental Groups of Open $4$-Manifolds}
\author{Hongzhi Huang and Xian-Tao Huang}
\address{Hongzhi Huang \\ Department of Mathematics \\ Jinan University\\ Guangzhou 510632}
\email{\href{mailto:huanghz@jnu.edu.cn}{huanghz@jnu.edu.cn}
}

\address{Xian-Tao Huang\\School of Mathematics\\  Sun Yat-sen University\\ Guangzhou 510275}
\email{\href{mailto:hxiant@mail2.sysu.edu.cn}{hxiant@mail2.sysu.edu.cn}}

\newtheorem{thm}{Theorem}[section]

\newtheorem{lem}[thm]{Lemma}
\newtheorem{slem}[thm]{Sublemma}

\newtheorem{cor}[thm]{Corollary}

\newtheorem{claim}[thm]{Claim}

\newtheorem{mainthm}{Theorem}[section] 
\newaliascnt{myMainThm}{mainthm}        
\newtheorem*{myMainThm*}{Theorem \Alph{mainthm}} 

\newenvironment{maintheorem}[1][]
{\begin{mainthm}[#1]}
	{\end{mainthm}}

\theoremstyle{definition}
\theoremstyle{remark}

\newtheorem{defn}[thm]{Definition}

\numberwithin{equation}{section}

\newcommand {\diam }{\mathrm{diam}}
\newcommand {\Isom }{\mathrm{Isom}}

\newcommand {\vol }{\mathrm{vol}}

\newcommand {\id}{\mathrm{id}}

\newcommand {\Z}{\mathbb{Z}}

\newcommand {\R}{\mathbb{R}}

\newcommand {\GH}{\xrightarrow{GH}}

\newcommand {\pGH}{\xrightarrow{GH}}

\newcommand {\spa}[1]{\langle{#1}\rangle}

\newcommand {\Ric}{\mathrm{Ric}}

\newcommand {\RCD}{\mathrm{RCD}}

\newcommand{\XXint}[3]{{
		\setbox0=\hbox{$#1{#2#3}{\int}$}
		\vcenter{\hbox{$#2#3$}}\kern-.5\wd0}}

\begin{document}

\maketitle
\begin{abstract}

Let $M$ be a 4-dimensional open manifold with nonnegative Ricci curvature. In this paper, we prove that if the universal cover of $M$ has Euclidean volume growth, then the fundamental group $\pi_1(M)$ is finitely generated. This result confirms Pan-Rong's conjecture \cite{PR18} for dimension $n = 4$. Additionally, we prove that there exists a universal constant $C>0$ such that $\pi_1(M)$ contains an abelian subgroup of index $\le C$. More specifically, if $\pi_1(M)$ is infinite, then $\pi_1(M)$ is a crystallographic group of rank $\le 3$. If $\pi_1(M)$ is finite, then $\pi_1(M)$ is isomorphic to a quotient of the fundamental group of a spherical 3-manifold.

	

\end{abstract}

\setcounter{tocdepth}{1}
\tableofcontents

\section{Introduction}\label{section1}

In this paper, we investigate the fundamental group of $4$-dimensional open manifolds with nonnegative Ricci curvature, whose universal cover has Euclidean volume growth.

Recall that, for an open $n$-manifold $M$ with nonnegative Ricci curvature, the following limit always exists and does not depend on the choice of $p\in M$,
	\begin{equation}\label{Def:EuclideanVolumeGrowth}
		\lim_{r\to\infty}\frac{\vol(B_r(p))}{r^n}=\nu,
	\end{equation}
	where $\vol(B_r(p))$ stands for the volume of the $r$-geodesic ball around $p$. If $\nu>0$ in (\ref{Def:EuclideanVolumeGrowth}), then we say that $M$ has Euclidean volume growth.
In this case, according to Cheeger-Colding's theory (\cite{CC96}), any asymptotic cone of $M$ is a metric cone over a length metric space $Z$, denoted by $C(Z)$, with its reference point as a cone vertex and the diameter of $Z$ satisfying $\diam(Z) \le \pi$.

The first main result of this paper is about the finite generation of fundamental groups.	In \cite{Mi68}, Milnor conjectured that the fundamental group of an open manifold with nonnegative Ricci curvature is finitely generated. Currently, this conjecture has been affirmed for dimensions $ n = 2 $ and $ 3 $ (\cite{CV35,Liu13,Pan20}), while it has been shown to be false for dimensions $n \geq 6 $ (\cite{BNS23,BNS23-II}). Therefore, the only unresolved cases remain in dimensions $ n = 4 $ and $ 5 $. From another perspective, it is also worthwhile to explore the verification of Milnor's conjecture under appropriate additional restrictions. For progress in this area, we refer the reader to the survey presented in \cite{Pan20a}. In \cite{PR18}, Pan and Rong conjectured that if the universal cover has Euclidean volume growth, then Milnor's conjecture holds.
Pan-Rong's conjecture was restated in \cite{BNS23-II} and it remains open to this day. Our first result can be regarded as a verification of the Pan-Rong's conjecture in the case of dimension $n = 4$.

\begin{maintheorem}\label{thm-finite generation}
	Let $M$ be an open 4-manifold with nonnegative Ricci curvature. If the Riemannian universal cover $\tilde{M}$ has Euclidean volume growth, then $\pi_{1}(M)$ is finitely generated.
\end{maintheorem}

We say that a metric space $X$ splits an $\R^k$-factor if it is isometric to a metric product of the form $\R^k \times Z$. The high-dimensional version of Theorem \ref{thm-finite generation} is as follows.

\begin{thm}\label{thm-finite generation'}
	Let $M$ be an open $n$-manifold with nonnegative Ricci curvature ($n\ge 4$). If the Riemannian universal cover $\tilde{M}$ has Euclidean volume growth, and every asymptotic cone of $M$ splits an $\R^{n-4}$-factor, then $\pi_{1}(M)$ is finitely generated.
\end{thm}

The second main result of this paper concerns the virtual abelianness of the fundamental groups.
In general, the fundamental group of an open manifold with nonnegative Ricci curvature is known to be virtually nilpotent (\cite{Mi68,Gr81,KW11}).
On the other hand, according to the work of Wei (\cite{Wei88}) and Wilking (\cite{Wil00}), it is known that any finitely generated virtually nilpotent group can be realized as the fundamental group of some open manifold with positive Ricci curvature. However, this is not the case for nonnegative sectional curvature, where the fundamental groups are always virtually abelian. This leads to an interesting question: what sufficient conditions can distinguish open manifolds with virtually abelian fundamental groups from the general open manifolds with nonnegative Ricci curvature? This topic has been deeply explored by Pan, with references including \cite{Pan21, Pan22, Pan22-II, NPZ24}, among others. Our second result asserts that $\pi_1(M)$ in Theorem \ref{thm-finite generation} is virtually abelian.

\begin{maintheorem}\label{thm-Almost Abelian}
	There exists a constant $C>0$ to the following effect. Let $M$ be an open 4-manifold with nonnegative Ricci curvature. Suppose the Riemannian universal cover $\tilde{M}$ has Euclidean volume growth.
	\begin{enumerate}
		\item If $\pi_1(M)$ is infinite, then $\pi_1(M)$ is isomorphic to a crystallographic group of rank $k$ with $1\leq k\leq3$. In particular $\pi_{1}(M)$ contains a free abelian normal subgroup $\mathbb{Z}^{k}$ of index $\leq C$.
		
		\item If $\pi_1(M)$ is finite, then $\pi_1(M)$ is isomorphic to a quotient of the fundamental group of a spherical $3$-manifold.
		In particular, $\pi_{1}(M)$ contains a normal cyclic subgroup of index $\le C$.
	\end{enumerate}

\end{maintheorem}

In \cite{Pan22-II}, Pan proved that for an open $n$-manifold $M$ with nonnegative Ricci curvature, if the universal cover $\tilde M$ has Euclidean volume growth and $\pi_1(M)$ has a non-maximal escape rate, then $\pi_1(M)$ contains an abelian subgroup whose index is bounded by a finite number that depends on the volume growth rate of $\tilde M$. In contrast, in Theorem \ref{thm-Almost Abelian}, the index of the abelian subgroup is bounded by a universal constant independent of the volume growth rate of $\tilde M$.

A well-known result, independently established by Li \cite{L86} and Anderson \cite{An90-II}, states that if an open $n$-manifold $M$ with nonnegative Ricci curvature has Euclidean volume growth, then its fundamental group is finite. Consequently, the condition that the universal cover $\tilde{M}$ has Euclidean volume growth, combined with condition (2) of Theorem \ref{thm-Almost Abelian}, is equivalent to the manifold $M$ itself having Euclidean volume growth. Theorem \ref{thm-Almost Abelian} (2) further shows that, when $n=4$, the fundamental group $\pi_1(M)$ is not only finite but also virtually cyclic.

A key ingredient in proving our main theorems is the recent breakthrough, by Bru\`e-Pigati-Semola \cite{BPS24}, in understanding lower dimensional non-collapsed Ricci limit spaces and $\mathrm{RCD}$-spaces.
One of their theorems assert that for $M$ as in case (2) of Theorem \ref{thm-Almost Abelian}, there exists a finite subgroup $\Gamma < O(4)$ that acts freely on $S^3$, such that any asymptotic cone of $M$ is homeomorphic to $C(S^3/\Gamma)$, the metric cone over the spherical space form $S^3/\Gamma$. Our proof involves showing that $\pi_1(M)$ is a quotient of $\Gamma$. However, $\pi_1(M)$ does not necessarily coincide with $\Gamma$. For example, consider the Eguchi-Hanson space, a complete Ricci-flat 4-manifold. It is diffeomorphic to the cotangent bundle of $S^2$ and has Euclidean volume growth. Its asymptotic cone is $\R^4 / \Z_2 = C(S^3 / \Z_2)$, where the cross section has a fundamental group $\Z_2$ different from $\pi_1(M)=\{e\}$.

A generalized version of Theorem \ref{thm-Almost Abelian} for higher dimensions $n\ge 4$ is as follows.

\begin{thm}\label{thm-Almost Abelian'}
	There exists a constant $C>0$ to the following effect. Let $M$ be an open $n$-manifold with nonnegative Ricci curvature $(n\ge 4)$. If the Riemannian universal cover $\tilde{M}$ has Euclidean volume growth, and,
	\begin{enumerate}
		\item If $\pi_1(M)$ is infinite and every asymptotic cone of $M$ splits an $\R^{n-4}$-factor, then $\pi_1(M)$ is isomorphic to a crystallographic group of rank $k$ with $1\leq k\leq3$. In particular $\pi_{1}(M)$ contains a free abelian normal subgroup $\mathbb{Z}^{k}$ of index $\leq C$.
		
		\item If $\pi_1(M)$ is finite and one asymptotic cone of $M$ splits an $\R^{n-4}$-factor, then $\pi_1(M)$ is isomorphic to a quotient of the fundamental group of a spherical $3$-manifold.
		In particular, $\pi_{1}(M)$ contains a normal cyclic subgroup of index $\le C$.
	\end{enumerate}
	
\end{thm}
We note that in Theorem \ref{thm-Almost Abelian'}, case (1) assumes that every asymptotic cone of $M$ splits an $\R^{n-4}$-factor, while case (2) only requires a single asymptotic cone splits an $\R^{n-4}$-factor. Neither of these assumptions ensures that $M$ itself splits an $\R^{n-4}$-factor.

The following corollary arises as a result of proving Theorem \ref{thm-finite generation} and Theorem \ref{thm-Almost Abelian}. It characterizes specific geometric properties that serve as sufficient conditions for the finite generation or virtually abelian nature of fundamental groups.
 
\begin{cor}\label{cor-Polarness}
	Let $M$ be an open 4-manifold with nonnegative Ricci curvature. If the Riemannian universal cover $\tilde{M}$ has Euclidean volume growth, then
	\begin{enumerate}
		\item The reference point of every asymptotic cone of $M$ is a pole.
		\item $\pi_1(M)$ has a vanishing escape rate.
	\end{enumerate}
\end{cor}

In a metric space $X$, a point $x \in X$ is called a pole of $X$ if, for every $y \in X \setminus \{x\}$, there exists a ray starting from $x$ that passes through $y$. According to Sormani's pole group theorem \cite{Sor99}, condition (1) of Corollary \ref{cor-Polarness} is a sufficient condition for the finite generation of $\pi_1(M)$. The concept of escape rate was introduced in \cite{Pan21}, where it was shown that condition (2) of Corollary \ref{cor-Polarness} also provides a sufficient condition for the virtual abelianness of $\pi_1(M)$.

We give a brief explanation of the approach to Theorem \ref{thm-finite generation}.

The key ingredient in our proofs is a topological stability result for good level sets of codimension $2$ almost splitting maps, which is due to \cite{BPS24}. More specifically, if we have a convergent sequence, in the pointed Gromov-Hausdorff sense, of $4$-manifolds with nonnegative Ricci curvature, $(M_i,p_i)\pGH(\R\times C(Z),p_\infty)$, then by Cheeger-Colding's theory, for any large $i$, there exist an $\epsilon_i$-splitting function $v_i:B_{100}(p_i)\to\R$ ($\epsilon_i\to0$) and a Green-type distance function $b_i:B_{100}(p_i)\to\R$ which uniformly converge to the projection $P_1:\R\times C(Z)\to\R$ and the distance function of $p_\infty$ respectively.
Bru\`e-Pigati-Semola \cite{BPS24} proved that there exist many $(t_i,s_i)\in[-1,1]\times[8,9]$, such that the level sets $\Sigma_i:=(v_i,\sqrt{b_i^2-v_i^2})^{-1}(t_i,s_i)$ are all homeomorphic to $Z$. We can choose $t_i$ to be the regular value of $u_i$. Hence $\Sigma_i$ is a boundary of a $3$-manifold. This implies that each $Z_i$ is homeomorphic to $S^2$. 

Building upon the aforementioned result, we show that if $\pi_1(M)$ is nilpotent, then it must be torsion-free. This significantly simplifies the structure of the equivariant asymptotic cones of $(\tilde{M}, \pi_1(M))$. Under our assumptions, we then employ the critical scaling argument, a technique developed by Pan, to prove the existence of a group $K$ such that for any $(X, x^*, G) \in \Omega(\tilde{M}, \pi_1(M))$, the isotropy group of $G$ at $x^*$ is isomorphic to $K$. Here, $\Omega(\tilde{M}, \pi_1(M))$ denotes the union of all equivariant asymptotic cones of $(\tilde{M}, \pi_1(M))$. The remaining proofs involve analyzing all possibilities of $(X,x^*,G)\in \Omega(\tilde{M}, \pi_1(M))$ to show that the orbit $G(x^*)$ is connected. This is sufficient to prove $\pi_1(M)$ is finitely generated.

We organize this paper as follows. Section \ref{section2} presents several simple results that will be frequently used throughout this paper. In Section \ref{section3}, we establish the first key lemma: if an open $n$-manifold $M$ with nonnegative Ricci curvature has Euclidean volume growth and one of its asymptotic cones splits an $\R^{n-3}$-factor, then $M$ is simply connected. Section \ref{section4} is devoted to proving a stability lemma regarding the isotropy groups at the reference point of limit groups. This is the second key lemma in the proof of Theorem \ref{thm-finite generation}. In Section \ref{section5}, we prove Theorem \ref{thm-finite generation}, which asserts the finite generation of $\pi_1(M)$, by examining all possible equivariant asymptotic cones of $\tilde M$. In Section \ref{section6}, we prove Theorem \ref{thm-Almost Abelian'} and Corollary \ref{cor-Polarness}, thereby proving Theorem \ref{thm-Almost Abelian} as well. Section \ref{section7} serves as an appendix, where we have compiled the proofs of several lemmas referenced in the preceding text.

\vspace*{12pt}



\section{Some Preliminary Lemmas}\label{section2}

In this section, we collect several useful lemmas that will be frequently used in the proofs of our main theorems.

Unless otherwise stated, throughout this paper, $(M, p)$ represents a pointed open manifold with nonnegative Ricci curvature, and $(\tilde M, \tilde p) \to (M, p)$ denotes its Riemannian universal cover. The symbol $d$ always refers to the distance function on the relevant spaces.
$\Isom (M)$ denotes the isometry group of $M$.

Let $\Gamma$ be a closed subgroup of $\Isom (M)$. For any $s_i\to\infty$, there exists a convergent subsequence $(s_i^{-1}M,p,\Gamma)\pGH(X,x,G)$ in the equivariant Gromov-Hausdorff sense, where $G$ is a closed subgroup of $\Isom(X)$. Such $(X,x,G)$ is called an equivariant asymptotic cone of $(M,p,\Gamma)$. We use $\Omega(M,\Gamma)$ to denote the union of all equivariant asymptotic cone of $(M,\Gamma)$. For the definition of equivariant Gromov-Hausdorff convergence, we refer to \cite{FY92}. 

The following fact provides sufficient conditions for ensuring the splitting of asymptotic cones.

\begin{lem}\label{Lem17:28}
Let $(M, p)$ be a pointed open manifold with nonnegative Ricci curvature,

\begin{itemize}
	\item [(1)] Let $\Gamma$ be a closed subgroup of $\Isom (M)$. If $\Gamma$ is finitely generated and the orbit $\Gamma(p)$ is unbounded, then for any $(Y,y^*,G)\in\Omega(M,\Gamma)$, the orbit $G(y^*)$ is unbounded and not discrete.
	
	\item [(2)] If there exists a non-trivial torsion-free element $\gamma\in\Isom (M)$ such that $\spa{\gamma}$ is a closed subgroup of $\Isom (M)$, then for any $(Y,y^*,H)\in\Omega(M,\langle\gamma\rangle)$, $H$ contains a closed subgroup $L$, which is isomorphic to $\R^l$ for $l\ge 1$, such that the orbit $L(y^*)$ is homeomorphic to $\R^l$.
\end{itemize}

 In particular, in either case as above, if $M$ has Euclidean volume growth, then $Y$ splits an $\R$-factor.

\end{lem}

Lemma \ref{Lem17:28} is well-known among experts. We have included its proof in the appendix for the reader's convenience.

An observation by Pan is that if the fundamental group of $M$ is not finitely generated, then there exists an equivariant asymptotic cone, $(Y,y,G)$, of $(\tilde M,\tilde p,\Gamma)$, where $\Gamma=\pi_1(M)$ is the deck transformation, such that $G(y)$ is not connected (\cite[Lemma 2.5]{Pan20}). By combining this with \cite[Lemma 3.1]{H24}, we derive the following lemma which establishes a criterion guaranteeing that whenever all asymptotic cones of the manifold split off at least $k$-Euclidean factors, there necessarily exists some asymptotic cone admitting at least $(k+1)$ splittings.

\begin{lem}\label{splitting-lem}
	Let $(M,p)$ be an open manifold with nonnegative Ricci curvature. Suppose that every asymptotic cone of $M$ splits an $\R^k$-factor. If $\pi_{1}(M)$ is nilpotent and not finitely generated, and the universal cover $\tilde{M}$ has Euclidean volume growth.
	Then there exists an asymptotic cone of $M$ which splits an $\R^{k+1}$-factor.
	Consequently, any normal covering $\bar{M}$ of $M$ has an asymptotic cone splitting an $\R^{k+1}$-factor.
\end{lem}

\begin{proof}
	Fix a Gromov's short basis $\{\gamma_{1},\gamma_{2},\ldots\}\subset \pi_{1}(M, p):=\Gamma$ (see \cite{Gr82}).
	Denote by $r_{i}= |\gamma_{i}|$, the length of shortest presentation of $\gamma_{i}$. Up to a subsequence, we have the following commutative diagram of equivariant Gromov-Hausdorff convergence,
	
	\begin{equation*}\label{diag2.1}
		\xymatrix@C=2.5cm{
			(r_{i}^{-1}\tilde{M}, \tilde{p}, \Gamma) \ar[d]_{} \ar[r]^{GH} & (\R^{k}\times \R^s\times X, (0^{k},0^s,x^{*}), G) \ar[d]^{} \\
			(r_{i}^{-1}M, p) \ar[r]^{GH} & (\R^k\times Y,(0^k,y^{*})),  }
	\end{equation*}
	where $X$ contains no lines and $G$ acts on the $\R^k$-factor trivially. According to \cite[Lemma 2.5]{Pan20}, the obit $G((0^{k},x^{*}))$ is not connected.

	Since $\Gamma$ is nilpotent, the limit group $G$ is also nilpotent. And since $\tilde{M}$ has Euclidean volume growth, $X$ is a metric cone with vertex $x^{*}$. As $X$ contains no lines, every isometry of $X$ must fix the vertex $x^{*}$. Combining with the non-connectedness of $G((0^{k},x^{*}))$, by \cite[Lemma 3.1]{H24}, we deduce that any tangent cone, $T_{y^*}Y$, of $Y$ at $y^{*}$ splits an $\R$-factor. Hence after passing to a subsequence, there exists $s_i\to\infty$, such that $s_ir_i^{-1}\to 0$, and $$(s_ir_i^{-1}M,p)\GH(\R^k\times T_{y^*}Y,(0^k,y'))=(\R^{k+1}\times Y_1,(0^{k+1},y')),$$which is the first conclusion. For the latter conclusion, again passing to a subsequence the following diagram holds.
	
	\begin{equation*}
		\xymatrix@C=2.5cm{
			(s_ir_i^{-1}\bar{M}, \bar{p}) \ar[d]_{} \ar[r]^{GH} & (\bar X, \bar x) \ar[d]^{\sigma} \\
			(s_ir_i^{-1}M, p) \ar[r]^{GH} & (\R^{k+1}\times Y_1,(0^{k+1},y')),  }
	\end{equation*}
	where $\sigma$ is a submetry. So the lines in $\R^{k+1}\times Y_1$ can be lifted through $\sigma$ to $\bar X$. Thus $\bar X$ splits an $\R^{k+1}$-factor.
\end{proof}

\section{A Simply Connectedness Lemma}\label{section3}

It is well known that if an open $3$-manifold with nonnegative Ricci curvature has Euclidean volume growth, then it is diffeomorphic to $\R^3$ \cite{Zh93,Liu13}. In particular, it is simply connected. However, this result does not extend to dimensions $4$ and higher. In this section, we prove that if an open $n$-manifold $M$ with nonnegative Ricci curvature has Euclidean volume growth ($n\ge 3$), and has one asymptotic cone splitting an $\R^{n-3}$-factor, then $M$ is simply connected. This is crucial for the proof of our main theorems.

\begin{thm}\label{SimplyConnectnessOf(n-3)symmetry}
	Let $N$ be an open $n$-manifold with nonnegative Ricci curvature which has Euclidean volume growth ($n\ge  3$).
	If there exists an asymptotic cone of $N$ which splits an $\R^{n-3}$-factor, then $N$ is simply connected.
\end{thm}

For our purposes, we require a slightly more general version of Theorem \ref{SimplyConnectnessOf(n-3)symmetry}, as stated below, whose proof follows the same approach as that of \cite[Lemma 9.12]{BPS24}, which corresponds to the case of $n=3$ in Theorem \ref{SimplyConnectnessOf(n-3)symmetry}. The main technical tools are  \cite[Theorem 5.2]{BPS24}, an annular version of the slicing theorem by Cheeger-Naber \cite{CN15}, and \cite[Proposition 7.1]{BPS24}, a topological regularity of good level sets
of almost splitting maps. Consequently, we will only present a sketched proof here.

\begin{lem}\label{KeyLemma}
Let $(M_{i}, p_{i})\GH(X, p_{\infty})$ be a convergent sequence satisfying the followings.
\begin{enumerate}
	\item [(1)] Each $M_i$ is a complete $n$-manifold with nonnegative Ricci curvature and there exists $v>0$ such that for any $i$, $\vol(B_1(p_i))\ge v>0$.
	
	\item [(2)] For each $i$, $\pi_i:(\check M_i,\check p_i)\to(M_i,p_i) $ is a normal covering with deck transformation group $\Gamma_i$ satisfying that $\diam(\Gamma_i(\check p_{i}))\to0$ as $i\rightarrow0$.
	
	\item [(3)] $(X,p_\infty)$ is isometric to $(\R^{n-3}\times C(Z),(0^{n-3},z^*))$ where $z^*$ is a vertex of $C(Z)$ and $Z$ is a length metric space with $\diam (Z)\le\pi$.

\end{enumerate}

Then for every large $i$, $\Gamma_{i}$ is trivial.
\end{lem}

\begin{proof}[Sketched Proof]
Up to a subsequence, we may assume $(\check M_i,\check p_i,\Gamma_i)\GH(Y,\check p_\infty,G)$ and $\pi_i:(\check M_i,\check p_i)\to(M_i,p_i)$ converges to a submetry $\pi:(Y,\check p_\infty)\to(X,p_\infty)$. Note that, by $\vol(B_{1}(\check{p}_{i}))\ge\vol (B_1(p_i))\ge v$, according to Cheeger-Colding's theory \cite{CC96}, every tangent cone of $Y$ is a metric cone. By allowing a slow blow-up, we can assume without losing our conditions that $Y$ itself is a metric cone with vertex $\check{p}_\infty$. By lifting the $\R^{n-3}$-factor of $X$ through $\pi$, we conclude that $(Y,\check p_\infty)$ is isometric to $(\R^{n-3}\times C(\check{Z}),(0^{n-3},\check z^*))$ for a length metric space $\check Z$ with $\diam (\check Z)\le \pi$. So the following commutative equivariant diagram holds,

\begin{equation}\label{diag19}
	\xymatrix@C=2.5cm{
		(\check{M}_i, \check{p}_i,\Gamma_i) \ar[d]_{\pi_i} \ar[r]^{GH} &\left(\R^{n-3}\times C(\check Z),(0^{n-3},\check z^*),G\right) \ar[d]^{\pi} \\
		(M_i, p_i) \ar[r]^{GH} &
		\left(\R^{n-3}\times C(Z),(0^{n-3},z^*)\right),  }
\end{equation}
where the limit group $G$ acts on $\R^{n-3}$-factor trivially. Because $\diam(\Gamma_i(\check p_{i}))\to0$, we have that $G$ fixes $(0^{n-3},\check z^*)$. Further, we have the following simple facts.
\customitemize{F}
	\item\label{itm:F1}	$B_{r}(\check{p}_{i})\subseteq \pi_{i}^{-1}(B_{r}({p}_{i}))\subseteq B_{r+\epsilon_i}(\check{p}_{i})$, where $\epsilon_i\to0$.
 	\item\label{itm:F2}  $G$ is a closed subgroup of $\Isom (\check Z)$ and for any $(x,t,\check z)\in\R^{n-3}\times C(\check Z)$, $\pi(x,t,\check z)=(x,t,z)$ for some $z\in Z$.
\end{enumerate}

According to \cite[Theorem 8.1]{BPS24}, $Z$ and $\check{Z}$ are both homeomorphic to the $2$-sphere $S^2$. For our goal, we make the following claim:
\begin{claim}\label{claim3.4}
After passing to a subsequence, for each large $i$,
\customitemize{c}
	\item \label{itm:c1} There exists an embedding submanifold $\Sigma_i\subset B_9(p_i)$, which is homeomorphic to the $2$-sphere. Further, there exists $(x_\infty,y_\infty)\in B_1(0^{n-3})\times[8,9]$, $\Sigma_i$ converges to $\Sigma_\infty:=\{(x_\infty,y_\infty,z)\in\R^{n-3}\times C(Z)|z\in Z\}$ with respect to the sequence (\ref{diag19}).

	\item \label{itm:c2} $\check\Sigma_i:=\pi_i^{-1}(\Sigma_i)$ is connected.
\end{enumerate}
\end{claim}
Assuming (\ref{itm:c1}) and (\ref{itm:c2}), and combining the elementary fact that the restriction map $\pi_i : \check{\Sigma}_i \to \Sigma_i \approx S^2$ is a covering map, it follows that $\pi_i$ is a one-sheeted cover, which yields the desired conclusion. So what remains to be proven is Claim \ref{claim3.4}.

In the following discussion, $\{\epsilon_i\}$ denotes a positive sequence that converges to $0$ as $i \to \infty$. Note that the specific value of $\epsilon_i$ is not essential, which may increase from line to line without explicit mention. The key point is that $\epsilon_i \to 0$ as $i \to \infty$.

The proof of claim (\ref{itm:c1}) follows verbatim from the one of \cite[Theorem 8.1]{BPS24}, as we are in the same situation. Here, we only provide a sketch of the construction of $\Sigma_i$ in (\ref{itm:c1}).

For every $r > s > 0$, define $A_{r,s}(p_i) := B_r(p_i) \setminus \bar{B}_s(p_i)$. According to Theorem 4.6 of \cite{BPS24}, for every large $i$, there exists a good Green distance with center $p_{i}$, denoted by $b_{i}:B_{20}(p_i)\to[0,\infty)$, satisfying
\begin{equation}\label{0:34}
	\sup_{B_{20}(p_{i})}|b_{i}-d_{p_{i}}|\leq \epsilon_i,
\end{equation}
where $d_{p_i}$ stands for the distance function with respect to $p_i$. In another aspect, by Cheeger-Colding's theory (see \cite{CC96}, \cite{CN15} etc.), for every sufficiently large $i$, there exists an $\epsilon_i$-splitting map $v_i=(v_i^1,\ldots,v_i^{n-3}) : (B_{100}(p_{i}),p_i)\rightarrow (\R^{n-3},0^{n-3})$ such that
\begin{equation}\label{0:35}
	v_i \text{ uniformly converges to the projection } P_1:\R^{n-3}\times C(Z)\to\R^{n-3}.
\end{equation}

On the set $\{b_{i}> |v_{i}|\}$, let $u_{i} :=\sqrt{b_{i}^{2}-v_{i}^{2}}$ and $A_i:=A_{9,8}(p_i) \cap \{|v_{i}| < 1\}$. Combining (\ref{0:34}) and (\ref{0:35}), $u_i:A_i\to\R$ uniformly converges to the distance function $d_{\R^{n-3}}$ with respect to the subset $\R^{n-3}\times \{z^*\}\subset \R^{n-3}\times C(Z)$.

By the slice theorem \cite[Theorem 5.2]{BPS24}, there exist a positive constant $c_{0}$ (depending on $X$ and independent of $i$) and Borel sets $\mathcal{B}_{i}\subset B_{1}(0^{n-3})\times[0, 10]$ such that
\customitemize{s}
	\item\label{itm:s1} the Hausdorff measure $\mathcal{H}^{n-2}(\mathcal{B}_{i}) \leq\epsilon_i$;
	\item\label{itm:s2} for every $(x, y)\in B_{1}(0^{n-3})\times[0, 10]\setminus \mathcal{B}_{i}$ with $8\leq \sqrt{|x|^{2}+y^{2}}\leq 9 $, the level set $\{(v_{i}, u_{i}) = (x, y)\}$ is not empty;
	\item\label{itm:s3} for every $s\in(0, c_{0})$ and every $q\in A_{i}$ with $(v_{i}, u_{i})(q) \notin \mathcal{B}_{i}$, there exists a lower triangular $(n-2)\times(n-2)$ matrix $L_{q,s}$ with positive diagonal entries such that $L_{q,s}\circ(v_{i}, u_{i}) : B_{s}(q)\rightarrow\R^{n-2}$ is an $\epsilon_i$-splitting map.
\end{enumerate}

By (\ref{itm:s1}), (\ref{itm:s2}) and Sard's theorem, for each $i$, we choose $(x_i,y_i)\in B_{0.9}(0^{n-3})\times[0,10]\setminus\mathcal{B}_i$ with $8.1\leq \sqrt{|x_{i}|^{2}+y_{i}^{2}}\leq 8.9$ such that $\Sigma_i:=\{(v_i,u_i)=(x_i,y_i)\}\neq\emptyset$ and $(x_i,y_i)$ is a regular value of $(v_i,u_i)$. Note that by (\ref{0:34}), for each large $i$, $\Sigma_i$ is contained in $A_i$. By (\ref{itm:s3}) and \cite[Proposition 7.1]{BPS24}, $\Sigma_i$ is a uniformly locally contractible $2$-dimensional manifold, uniformly also in $i$. And $\Sigma_i$ is the boundary of the $3$-dimensional compact submanifold $v_i^{-1}(x_i)\cap \{u_i\le y_i\}$. Up to a subsequence, we may assume $(x_i,y_i)\to(x_\infty,y_\infty)\in B_1(0^{n-3})\times [0,10]$ with $8\le \sqrt{|x_\infty^2|+y_\infty^2}\le 9 $. Then (\ref{itm:c1}) follows from the same argument as in the proof of \cite[Theorem 8.1]{BPS24}.

Now we proceed to prove (\ref{itm:c2}). Argue by contradiction. We assume that, for every $i$, $\check \Sigma_i$ is not connected. For a fixed $i$, we can express it as $\check\Sigma_i = \sqcup_{j=1}^k S_j$, where $S_1, \ldots, S_k$ are the mutually distinct connected components of $\check\Sigma_i$ for some $k>1$. Let $s_i:=\min\{d(S_{j_1},S_{j_2})|j_1\neq j_2\}>0$. By a re-index, we may assume that there exist $q_{i,1}\in S_1$, $q_{i,2}\in S_2$ such that $d(q_{i,1},q_{i,2})=s_i$.

Observe that
\begin{equation}\label{s_i}
	\lim_{i\to\infty}s_i=0.
\end{equation}
If not, up to a subsequence, we assume $s_i\ge \eta >0$. Referring to the diagram (\ref{diag19}), based on (\ref{itm:c1}) and facts (\ref{itm:F1}) and (\ref{itm:F2}), we observe that
$\check\Sigma_{i}\subset B_{10}(\check p_i)$ converges to $\check \Sigma_\infty:=\{(x_\infty,y_\infty,\check z)\in\R^{n-3}\times C(\check Z)|\check z\in \check Z\}$. Thus passing to a subsequence again, we may assume $q_{i,j}\to q_j\in\check Z_\infty$ as $i\to\infty$, $j=1,2$, with $q_1\neq q_2$. Note that $\check Z_\infty$ is an $\mathrm{RCD}(1,2)$-space, thus there exists a midpoint, $q_3$, of $q_1,q_2$ in $ \check Z_\infty$. By the cosine law, $d_Y(q_1,q_3)=d_Y(q_2,q_3)<d_Y(q_1,q_2)$, where $d_Y$ is the distance on $Y=\R^{n-3}\times C(\check Z)$. Again using the fact $\check\Sigma_i\to\Check\Sigma_\infty$, there exist $\check\Sigma_i\ni q_{i,3}\to q_3$. So for every large $i$, $\max\left(d(q_{i,1},q_{i,3}),d(q_{i,2},q_{i,3})\right)<d(q_{i,1},q_{i,2})=s_i$.
By the definition of $s_i$, $q_{i,3}\in S_1\cap S_2$, which is impossible.

Let $\check{v}_i:=v_i\circ\pi_i:B_{100}(\check p_i)\to\R^{n-3}$, $\check{b}_i:=b_i\circ \pi_i:B_{20}(\check p_i)\to[0,\infty)$ and $\check{u}_i:=u_i\circ \pi_i:\pi_i^{-1}(A_i)\to\R$. Note that
\begin{equation}\label{15:18}
	\check\Sigma_i=\{(\check v_i,\check u_i)=(x_i,y_i)\},
\end{equation}
where the right inclusion is obvious and the left inclusion follows from (\ref{itm:F1}) and the fact that $\Sigma_i\subset B_9(p_i)$. Combining (\ref{15:18}), (\ref{itm:s3}), the choice of $x_i,y_i$, and the covering lemma \cite[Lemma 1.6]{KW11}, we conclude that, for every large $i$, and any $s\in(0, c_{0})$, there exists a lower triangular $(n-2)\times(n-2)$ matrix $L_{q_{i,1},s}$ with positive diagonal entries such that $L_{q_{i,1},s}\circ(\check v_{i}, \check u_{i}) : B_{s}(q_{i,1})\rightarrow\R^{n-2}$ is an $\epsilon_i$-splitting map. We apply \cite[Proposition 7.1 (ii)]{BPS24} with $\omega_i:=(\check v_i,\check u_i)-(x_i,y_i):(B_{\frac12c_0}(q_{i,1}),q_{i,1})\to(\R^{n-2},0^{n-2})$ to conclude that there exists universal $\overline C>0$, $r(n,v)>0$, for any $r\le r(n,v)$, $\check q\in \omega_i^{-1}(0^{n-2})\cap B_{r}(q_{i,1})$, there exists a curve $\sigma:[0,1]\to\omega_i^{-1}(0^{n-2})\cap B_{\overline Cr}(q_{i,1})$ from $q_{i,1}$ to $\check q$. Now, by (\ref{s_i}), for sufficiently large $i$, one can choose $\check{q} := q_{i,2}$ and $r = 2s_i$, which yields a contradiction to the choice of $q_{i,1}$ and $q_{i,2}$. The proof is thus complete.

\end{proof}

Theorem \ref{SimplyConnectnessOf(n-3)symmetry} is an immediate corollary of the above lemma.

\begin{proof}[Proof of Theorem \ref{SimplyConnectnessOf(n-3)symmetry}]
Take $p\in N$ and let $(\tilde{N},\tilde{p})\xrightarrow{\pi}(N, p)$ be the universal cover with deck transformation $\Gamma=\pi_{1}(N, p)$. By \cite{L86} (or \cite{An90-II}), $\Gamma$ is a finite group.
Note that $\diam(\pi^{-1}(p))$ is finite.

By the assumption and Cheeger-Colding's theory, there exists a sequence $r_{i}\rightarrow \infty$ such that
\begin{align*}
(r_{i}^{-1}N, p)\xrightarrow{GH}(\R^{n-3}\times C(Z), (0,z^{*})).
\end{align*}
Obviously, $\diam^{r_i^{-1}N}(\Gamma(\tilde p))\rightarrow 0$. By Lemma \ref{KeyLemma}, we conclude that $\Gamma$ is trivial, i.e. $N$ is simply connected.
\end{proof}

\section{A Stability Lemma for Isotropy Groups}\label{section4}

In this section, we establish a lemma asserting that, under certain restrictions, the isomorphism type of the isotropy subgroups of limit groups of equivariant asymptotic cones at the reference point is uniquely determined. This lemma is used to prove Theorem \ref{thm-finite generation}.

Throughout this section, let $(M,p,\Gamma)$ be an open $n$-manifold with nonnegative Ricci curvature, and $\Gamma$ be a closed subgroup of $\Isom(M)$. Let $\Omega(M,\Gamma)$ be the union of all equivariant asymptotic cones of $(M,\Gamma)$.
We further assume that $M$ has Euclidean volume growth.

For simplicity of notions, for every $(X,x^*,G)\in \Omega(M,\Gamma)$, let $I(G)$ denote the isotropy of $G$ at the reference point $x^*$, and $CS(X)$ denote the unit cross section of $X$ equipped with the intrinsic metric.
Hence $I(G)$ is a closed subgroup of $\Isom (CS(X))$.
According to Colding-Naber \cite{CN12}, $G$ is a Lie group, and so is $I(G)$.
By the $\mathrm{RCD}$-theory, $CS(X)$ is a non-collapsed $\mathrm{RCD}(n-2,n-1)$-space (see \cite{K15,DePGig18}).

The main goal of this section is the following.
\begin{lem}\label{StabilityOfIsotropy}
	If $\max\{\dim(I(G))|(X,x^*,G)\in \Omega(M,\Gamma)\}\le 1$, then there exists a compact Lie group $K$, such that for any $(X,x^*,G)\in \Omega(M,\Gamma)$, $I(G)$ is isomorphic to $K$.
\end{lem}
The proof of Lemma \ref{StabilityOfIsotropy} follows a similar idea as presented in \cite{Pan19}.

For any $(\R^k\times C(X),(0^k,x^*),G)\in \Omega(M,\Gamma)$, where $\diam(X)<\pi$, we use $(v,t,x)$ to denote an arbitrary point on $\R^k\times C(X)$, where $v\in\R^k$ and $t\in[0,+\infty)$ and $x\in X$.
Recall that $\Isom(\R^k\times C(X))=\Isom (\R^k)\times \Isom (X)$. Hence, for any $g\in G$, there exist $A\in O(k)$, $v_0\in\R^k$, $\mathbf{ g}\in\Isom(X)$ such that $g(v,t,x)=(Av+v_0,t,\mathbf {g}(x))$.

The following definition is borrowed from \cite{Pan19}.
\begin{defn}
	
For $(\R^k \times C(X), (0^k, x^*), G) \in \Omega(M, \Gamma)$, where $\diam(X) < \pi$, we say that $G$ satisfies property (P), if $G=T(G)I(G)$, where $T(G)$ is the subgroup of translations on the $\R^k$-factor of $G$. Specifically, for every $g \in G$ defined by $g(v, t, x) = (Av + v_0, t, \mathbf{g}(x))$, both $g_1(v, t, x) := (Av, t, \mathbf{g}(x))$ and $g_2(v, t, x) := (v + v_0, t, x)$ belong to $G$.
\end{defn}

\begin{lem}\label{ReduceToTangentCone}
	Let $(\R^k\times C(X),(0^k,x^*),H)$ be an equivariant tangent cone of $$(\R^k\times C(X),(0^k,x^*),G)\in\Omega(M,\Gamma)$$ at $(0^k,x^*)$, where $\diam (X)<\pi$. Then $H$ satisfies property (P). Furthermore, $I(H)$ is isomorphic to $I(G)$.
\end{lem}
\begin{proof}
	By assumption, there exists a sequence $r_i \to \infty$ such that
	\begin{equation}\label{TangentCone}
		(r_i(\R^k \times C(X)), (0^k, x^*), G) \GH (\R^k \times C(X), (0^k, x^*), H).
	\end{equation}

	Note that the standard retraction map, $\varphi_i:\R^k\times C(X)\to r_i(\R^k\times C(X))$ defined by $\varphi_i(v,t,x):=(r_i^{-1}v,r_i^{-1}t,x)$, is an isometry.
Therefore according to $\varphi_i$, $(r_i(\R^k\times C(X)),(0^k,x^*),G)$ is equivariantly isometric to  $(\R^k\times C(X),(0^k,x^*),r_iG)$, where the $r_iG$-action is defined to be: given any $g\in G$, for any $(v,t,x)\in \R^k\times C(X)$, $$(r_ig)(v,t,x):=(\varphi_i^{-1}\circ g\circ \varphi_i)(v,t,x)=r_i(g(r_i^{-1}v,r_i^{-1}t,x)),$$
where $r_i(v,t,x):=(r_iv,r_it,x)$.
Hence (\ref{TangentCone}) becomes,
	\begin{equation}\label{TangentConeConvergence}
		(\R^k\times C(X),(0^k,x^*),r_iG)\GH (\R^k\times C(X),(0^k,x^*),H).
	\end{equation}

	Fix an $h\in H$ with $h(v,t,x)=(Av+v_0,t,\mathbf{h}(x))$. Then there exists a sequence of $g_i\in G$, such that $r_ig_i\GH h$ with respect to (\ref{TangentConeConvergence}). For each $i$, let $g_i(v,t,x)=(A_iv+v_i,t,\mathbf{g}_i(x))$. So $(r_ig_i)(v,t,x)\GH h(v,t,x)$ implies $(A_iv+r_iv_i,t,\mathbf{g}_i(x))\to (Av+v_0,t,\mathbf{h}(x))$. This implies that $A_i\to A,r_iv_i\to v_0$ and $\mathbf{g}_i\to\mathbf{h}$. Specially, $v_i\to0$, and $g_i\to g_\infty$ in $\Isom (\R^k\times C(X))$, where $g_\infty(v,t,x):=(A v,t,\mathbf{h}(x))$. Since $G$ is closed, $g_\infty\in I(G)$.
We conclude that
$$(\R^k\times C(X),(0^k,x^*),r_ig_\infty,r_ig_ig_\infty^{-1})\GH(\R^k\times C(X),(0^k,x^{*}),h_1,h_2),$$ where $$h_1(v,t,x)=(A v,t,\mathbf{h}(x))$$ and $$h_2(v,t,x)=\lim_{i}(r_ig_ig_\infty^{-1})(v,t,x)=\lim_i(A_iA^{-1}v+r_iv_i,t,\mathbf{g}_i\mathbf{h}^{-1}(x))=(v+v_0,t,x).$$
Hence $h=h_2h_1$ with $h_1,h_2\in H$. This completes the proof of the first required conclusion. Note that when $v_0=0$, the above discussion defines an embedding $I(H)\to I(G)$ defined by $h=h_1\mapsto g_\infty$. This map is obvious surjective.
This proves the second conclusion.
\end{proof}

\begin{lem}\label{ConvergenceofCrossections}
	Suppose that $(X_i,x_i^*,G_i)\GH(X,x,G)$ is a convergent sequence in $\Omega(M,\Gamma)$. If every $G_i$ satisfies property (P), then $(CS(X_i),I(G_i))\GH(CS(X),I(G))$.
\end{lem}
\begin{proof}
	It is easy to see $(CS(X_i),I(G_i))\GH (CS(X), H)$ for some closed $H\subset I(G)$.
Hence we only need to prove $I(G)\subset H$.
For any $h_\infty\in I(G)$, take $h_i\in G_i$ such that $(X_i,x_i^*,h_i)\GH(X,x,h_\infty)$. Specially, there exists a sequence of $\epsilon_i\to0$ such that
	\begin{equation}\label{AlmostIsotropy}
		d(h_i(x_i^*),x_i^*)\le\epsilon_i.
	\end{equation}
	For each $i$, let $X_i=\R^{k_i}\times C(Z_i)$ where $\diam(Z_i)<\pi$ and $h_i(v,t,z)=(A_iv+v_i,t,\mathbf{h}_i(z))$. Since $G_i$ satisfies property (P), $h_{1,i}(v,t,z):=(A_iv,t,\mathbf{h}_i(z))$ and $h_{2,i}(v,t,z):=(v+v_i,t,z)$ are both in $G_i$. By (\ref{AlmostIsotropy}), $|v_i|\le\epsilon_i$.
Hence we have $$(X_i,x_i^*,h_{1,i})\GH(X,x,h_\infty).$$
	Note that $h_{1,i}\in I(G_i)$, so $$(CS(X_i),h_{1,i})\GH(CS(X),h_\infty),$$which implies $h_\infty\in H$.
This completes the proof.
\end{proof}

\begin{proof}[Proof of Lemma \ref{StabilityOfIsotropy}]

	We define $K$ as below. Choose $(Y,y^*,S)\in \Omega(M,\Gamma)$ such that
	\begin{enumerate}
		\item For any $(X,x^*,G)\in \Omega(M,\Gamma)$, $\dim(I(S))\le\dim(I(G))$,
		
		\item For any $(X,x^*,G)\in \Omega(M,\Gamma)$, if $\dim(I(S))=\dim(I(G))$, then $\#(I(S)/I(S)_0)\le \#(I(G)/I(G)_0)$, where $\#(I(G)/I(G)_0)$ stands for the number of connected components of $I(G)$.
	\end{enumerate}
Such $(Y,y^*,S)$ always exists because $\dim(I(G))$ and $\#(I(G)/I(G)_0)$ are nonnegative integers. Now define $K:=I(S)$. Note that by assumptions, we have $\dim K=0$ or $1$.
$K$ satisfies the following gap property.

	\begin{slem}\label{GapLemma}
		Under the assumption of Lemma \ref{StabilityOfIsotropy}, there exists an $\eta>0$ depending on $K$, to the following effect. For any $(X_j,x^*_j,G_j)\in\Omega(M,\Gamma)$, $j=1,2$, if $G_1$ satisfies property (P), and $I(G_1)$ is isomorphic to $K$ and $d_{GH}((X_1,x^*_1,G_1),(X_2,x^*_2,G_2))\le\eta$, then $I(G_2)$ is isomorphic to $K$.
	\end{slem}
	\begin{proof}[Proof of Sublemma \ref{GapLemma}]
		
		Argue by contradiction. Assume that there exists a sequence of $(X_{j,i},x^*_{j,i},G_{j,i})\in\Omega(M,\Gamma)$, $j=1,2$, such that every $G_{1,i}$ satisfies property (P), and $I(G_{1,i})$ is isomorphic to $K$ and $$d_{GH}((X_{1,i},x^*_{1,i},G_{1,i}),(X_{2,i},x^*_{2,i},G_{2,i}))\le\delta_i\to0,$$ while every $I(G_{2,i})$ is not isomorphic to $K$.

		Up to a subsequence, we have $$(X_{j,i},x_{j,i}^*,G_{j,i})\GH(X,x^*,G),$$for $j=1,2$.
		
		For $j=1$, according to Lemma \ref{ConvergenceofCrossections}, we have $$(CS(X_{1,i}),I(G_{1,i}))\GH(CS(X),I(G)).$$
		
Note that $CS(X_{1,i})$ are all non-collapsed $\mathrm{RCD}(n-2,n-1)$-spaces, and the above convergence is non-collapsed. By \cite[Lemma 3.2]{MRW08} (see also \cite[Corollary 1.10]{Pan18}) and \cite[Theorem 0.8]{PR18}, we conclude that for every large $i$, there exists an injective Lie group homomorphism $\phi_i:I(G_{1,i})\to I(G)$, which is an $\epsilon_i$-equivariant Gromov-Hausdorff approximation, where $\epsilon_i\rightarrow 0$.
So under the standard metric induced by $I(G)$-action, $\phi_i(I(G_{1,i}))$ is an $\epsilon_i$-dense subgroup in $I(G)$.
		
We claim that $I(G)$ is isomorphic to $K$. If $\dim K = 0$, then $I(G_{1,i})$ are all discrete and have uniformly bounded orders. Thus $I(G)$ is discrete. So for large $i$, $\phi_i$ is an isomorphism. If $\dim K = 1$, by our assumption, $\dim(I(G))=\dim(\phi_i(I(G_{1,i})))=1$. Combining the fact that $\phi_i(I(G_{1,i}))$ is $\epsilon_i$-dense in $I(G)$, for large $i$, $I(G)=\phi_i(I(G_{1,i}))$. Hence in both cases, $I(G)$ is isomorphic to $K$.

For $j=2$, we assume
		
		$$(CS(X_{2,i}),I(G_{2,i}))\GH(CS(X),H),$$ where $H$ is a closed subgroup of $I(G)$. And again by \cite{MRW08} and \cite{PR18}, for large $i$, there exists an injective Lie group homomorphism $\psi_i:I(G_{2,i})\to H$ which is an $\epsilon_i$-equivariant Gromov-Hausdorff approximation. This implies that $I(G_{2,i})$ is isomorphic to a subgroup of $K$. By the choice of $K$, $I(G_{2,i})$ is isomorphic to $K$, which is a contradiction.
	\end{proof}
	
	We continue the proof of Lemma \ref{StabilityOfIsotropy}. We argue by contradiction. Suppose that there exists $(X,x^*,G)\in \Omega(M,\Gamma)$ such that $I(G)$ is not isomorphic to $K$.
 Let $(Y,y^*,S)$ be defined as at the beginning of the proof.
 By Lemma \ref{ReduceToTangentCone}, up to a scaling to the tangent cone, we may further assume $(Y,y^*,S)$ satisfies property (P). The subsequent discussion is a critical scaling argument originally developed by Pan.
	
	We choose $r_i$, $s_i$ satisfying the following.
	\begin{enumerate}
		\item $r_i\to\infty$, $s_i\to\infty$, $r_i^{-1}s_i\to\infty$,
		\item $(r_i^{-1}M,p,\Gamma)\GH(X,x^*,G)$,
		\item $(s_i^{-1}M,p,\Gamma)\GH(Y,y^*,S)$.
	\end{enumerate}
	
	Let $\Lambda_i:=r_i^{-1}s_i$, $(N_i,p_i):=(s_i^{-1}M,p)$. Then the above properties become
	\begin{enumerate}
		\item  $\Lambda_i\to\infty$,
		\item $(\Lambda_iN_i,p_i,\Gamma)\GH(X,x^*,G)$,
		\item $(N_i,p_i,\Gamma)\GH(Y,y^*,S)$.
	\end{enumerate}
	
	Define
	\begin{align*}
		A_i:=&\{l\in [1,\Lambda_i]|d_{GH}((l N_i,p_i,\Gamma),(Z,z^*,H))\le0.1\eta,\\&\text{for some }(Z,z^*,H)\in\Omega(M,\Gamma)\text{ satisfying property (P) and }I(H)\cong K \},
	\end{align*}
	where $\eta$ is given in Sublemma \ref{GapLemma}.
	
	Note that for all large $i$, $1\in A_i$, and $\Lambda_i\notin A_i$. To see the latter one, if $\Lambda_i\in A_i$, then $$d_{GH}((X,x^*,G),(Z,z^*,H))\le0.2\eta,$$  for some $(Z,z^*,H)\in\Omega(M,\Gamma)$ satisfying property (P) and $I(H)\cong K$. Applying Sublemma \ref{GapLemma}, we conclude that $I(G)$ is isomorphic to $K$, which contradicts the assumption on $(X,x^*,G)$.
	
	Let $\lambda_i\in A_i$ such that $\sup A_i-i^{-1}< \lambda_i\le \sup A_i$.
	
	We claim that $\Lambda_i\lambda_i^{-1}\to\infty$. If not, up to a subsequence, we may assume $\Lambda_i\lambda_i^{-1}\to C\in[1,+\infty)$, then
	
	$$(\lambda_iN_i,p_i,\Gamma)\GH(C^{-1}X,x^*,G).$$
	Then by the definition of $\lambda_i$ and $A_i$, $$     d_{GH}((C^{-1}X,x^*,G),(Z,z^*,H))\le0.2\eta,$$ for some $(Z,z^*,H)\in\Omega(M,\Gamma)$ satisfying property (P) and $I(H)\cong K$. Applying Sublemma \ref{GapLemma}, we conclude that $I(G)$ is isomorphic to $K$, contradicting to the assumption on $(X,x^*,G)$. So the claim follows.

	Up to a subsequence, we assume
	$$(\lambda_iN_i,p_i,\Gamma)\GH(W,w^*,L).$$
 Then
 $$     d_{GH}((W,w^*,L),(Z,z^*,H))\le0.2\eta,$$ for some $(Z,z^*,H)\in\Omega(M,\Gamma)$ satisfying property (P) and $I(H)\cong K$.
 Applying Sublemma \ref{GapLemma}, we conclude that $I(L)$ is isomorphic to $K$. So by Lemma \ref{ReduceToTangentCone}, the equivariant tangent cone, $(T_{w^*}W,w',L')$, of $(W,w^*,L)$ satisfies property (P) and $I(L')\cong K$. By a standard diagonal argument, after passing to a subsequence of $i$, we may choose $R_i$ satisfying the following:
	\begin{enumerate}
		\item $R_i\to\infty$, $R_i\lambda_i<\Lambda_i$,
		\item $(R_i\lambda_iN_i,p_i,\Gamma)\GH(T_{w^*}W,w',L')$.
	\end{enumerate}
Hence for large $i$, $R_i\lambda_i\in A_i$. This implies $R_i\sup A_i-R_i\epsilon_i< R_i\lambda_i\le\sup A_i$, where $\epsilon_i$ is a subsequence of $i^{-1}$. So $\sup A_i\le \frac{R_i}{R_i-1}\epsilon_i\to0$, contradicting $\sup A_i\ge 1$.
\end{proof}

\section{Proof of Theorem \ref{thm-finite generation}: the Finite Generation}\label{section5}

In this section, we focus on proving Theorem \ref{thm-finite generation}. Since the proof of Theorem \ref{thm-finite generation'} follows the same reasoning as that of Theorem \ref{thm-finite generation}, we will only present the proof of Theorem \ref{thm-finite generation} for the sake of brevity. Our main effort will be directed toward proving the following theorem.

\begin{thm}\label{OrbitisConnected}
	Let $M$ be as in Theorem \ref{thm-finite generation}. If $M$ is not flat, and $\Gamma:=\pi_1(M)$ is infinite and abelian, then there exists an integer $l\in\{1,2\}$, such that for any $(X,x^*,G)\in\Omega(\tilde M,\Gamma)$, the orbit $G(x^*)$ is isometric to $\R^l$. 
\end{thm}
\begin{proof}[Proof of Theorem \ref{thm-finite generation} assuming Theorem \ref{OrbitisConnected}]
	Without loss of generality, we assume that $M$ is not flat. If $\Gamma$ is not finitely generated, according to Wilking's reduction \cite[Corollary 3.2]{Wil00}, then $\Gamma$ contains an abelian subgroup $\Gamma'$ that is also not finitely generated. Applying Theorem \ref{OrbitisConnected} to $\tilde M/\Gamma'$, it yields that for any $(X,x^*,G)\in\Omega(\tilde M,\Gamma')$, $G(x^*)$ is isometric to $\R^l$. However, by Pan's observation \cite[Lemma 2.5]{Pan20}, since $\Gamma'$ is not finitely generated, then there exists an $(X, x^*, G) \in \Omega(\tilde{M}, \Gamma')$ such that $G(x^*)$ is not connected, which is a contradiction. 
\end{proof}

Now we begin the proof of Theorem \ref{OrbitisConnected}. First, we have the following,

\begin{lem}\label{Torsion-free}
	$\Gamma$ is torsion-free.
\end{lem}
\begin{proof}
	Suppose on the contrary, there exists a nontrivial torsion element $\gamma\in \Gamma$.
	Then $N:=\tilde{M}/\langle\gamma\rangle$ has Euclidean volume growth and is not simply connected. 
	
	Note that there exists an asymptotic cone of $N$ splitting an $\R$-factor. Indeed, if $\Gamma$ is finitely generated, then $\Gamma/\spa{\gamma}$ contains a non-trivial torsion-free element. According to Lemma \ref{Lem17:28} (2), every asymptotic cone of $N$ splits an $\R$-factor. If $\Gamma$ is not finitely generated, according to Lemma \ref{splitting-lem}, there is an asymptotic cone of $N$ which splits an $\R$-factor.

	Then by Theorem \ref{SimplyConnectnessOf(n-3)symmetry}, $N$ is simply connected, which is a contradiction.
	Hence $\Gamma$ is torsion-free. 
\end{proof}

\begin{lem}
	Let $k(X)$ denote the dimension of the maximal Euclidean factor of $X$. Then, $k(X)\in\{1,2\}$ and $G$ contains a closed $\R$-subgroup.
\end{lem}
\begin{proof}
	According to the codimension $2$ regularity theorem by Cheeger-Colding (\cite{CC97}), if there exists an $(X,x^*,G)\in (\tilde M,\Gamma)$ such that $k(X)\ge 3$, then $X$ is isometric to $\R^4$. With this, by Colding's rigidity theorem on maximal volume growth (\cite{Col97}), $\tilde M$ itself is isometric to $\R^4$, which contradicts that $M$ is not flat.
	On the other hand, By Lemma \ref{Torsion-free}, $\Gamma$ contains a non-trivial torsion-free element. By Lemma \ref{Lem17:28} (2), for any $(X,x^*,G)\in (\tilde M,\Gamma)$, $G$ contains at least a closed $\R$-subgroup, and $k(X)\ge 1$.
	Thus, $k(X)\in\{1,2\}$.
\end{proof}

Put $K=P_1(G)$, where $P_1:\Isom(X)\to\Isom (\R^{k(X)})$ is the standard projection, which is a closed map. Hence $K$ is an abelian closed subgroup of $\Isom(\R^{k(X)})$ and it includes a closed $\R$-subgroup. The following lemma categorizes all possible cases of $K$.

\begin{lem}\label{T2T3}\quad
	
	\begin{enumerate}
		\item If $k(X)=1$, then $K$ is isomorphic to $\R$ which acts on the $\R^1$-factor of $X$ by translation.
		
		\item If $k(X)=2$, then $(X,x^*)$ is isometric to $(\R^2\times C(S_r^1),(0^2,z^*))$, where $S_r^1$ is the circle with radius $r\in(0,1)$, and $r$ only depends on the volume growth rate of $\tilde M$, and $z^*$ is the cone vertex of $C(S_r^1)$. We may choose a normal coordinate $(x,y)$ on the $\R^2$-factor, and for each $t\in(-\infty,+\infty)$, we define $t_1(x,y):=(x+t,y)$ and $t_2(x,y):=(x,y+t)$, $r_t(x,y):=(x,t-y)$. Then
		\begin{enumerate}
			\item If $G(x^*)$ is connected, then one of the following holds.
			\customitemize{a}

				\item \label{itm:a1} $K=\spa{t_1,s_2|t,s\in(-\infty,+\infty)}$.
				
				\item \label{itm:a2}  $K=\spa{t_1,r_0|t\in(-\infty,+\infty)}$.
				
				\item \label{itm:a0} $K=\spa{t_1|t\in(-\infty,+\infty)}$.
			\end{enumerate}
			
			\item  If $G(x^*)$ is not connected, then one of the following holds.
			\customitemize{b}
				\item \label{itm:b1} There exists $a>0 $ such that $K=\spa{t_1,a_2|t\in(-\infty,+\infty)}$.
				
				\item \label{itm:b2} There exists $b>0 $ such that $K=\spa{t_1,r_b|t\in(-\infty,+\infty)}$.
			\end{enumerate}
		\end{enumerate}

	\end{enumerate}

\end{lem}

	The proof of the above Lemma is carried out through an elementary yet detailed discussion. For completeness, we have included it in the appendix.

In fact, the various situations in Lemma \ref{T2T3} cannot coexist.

\begin{lem}\label{(a1)}
	If there exists $(X,x^*,G)\in\Omega(\tilde M,\Gamma)$ satisfying $k(X)=2$ and (\ref{itm:a1}), then every $(Y,y^*,H)\in\Omega(\tilde M,\Gamma)$ satisfies $k(Y)=2$ and (\ref{itm:a1}).
\end{lem}
\begin{proof}
	
Note that for any $(X_1, x^*_1, G_1) \in \Omega(\tilde{M}, \Gamma)$ satisfying
\begin{equation}\label{17:37}
	d_{GH}((X_1, x^*_1, G_1(x^*_1)), (Z, z^*, S(x^*))) \leq 0.01,
\end{equation}
for some $(Z, z^*, S) \in \Omega(\tilde{M}, \Gamma)$ such that $k(Z) = 2$ and (\ref{itm:a1}) holds, by Lemma \ref{T2T3}, it follows that $(X_1, x^*_1)$ is isometric to $(Z, z^*)$, and $G_1$ satisfies either condition (\ref{itm:a1}) or (\ref{itm:b1}) for $a \in (0, 0.03)$. If $G_1$ satisfies (\ref{itm:a1}), then the left-hand side of \eqref{17:37} equals $0$. If $G_1$ satisfies (\ref{itm:b1}) for $a \in (0, 0.03)$, then the left-hand side of \eqref{17:37} is strictly less than $a$.
	
The desired conclusion follows from the following sublemma by taking the limit as $\epsilon \to 0$.
\begin{slem}
	Under the assumption of Lemma \ref{(a1)}, for any $\epsilon\in(0,0.01)$, there exists $(Z,z^*,S)\in \Omega(\tilde M,\Gamma)$ satisfying $k(Z)=2$ and (\ref{itm:a1}) such that, $$d_{GH}((Y,y^*,H(y^*)),(Z,z^*,S(y^*)))\le \epsilon.$$ 
\end{slem}

\begin{proof}[Proof of the sublemma]
	
	Argue by contradiction. Suppose the opposite is true. That is, there exists $\epsilon\in(0,0.01)$, for any $(Z,z^*,S)\in \Omega(\tilde M,\Gamma)$ satisfying $k(Z)=2$ and (\ref{itm:a1}), we have 
	\begin{equation}\label{0:31}
		d_{GH}((Y,y^*,H(y^*)),(Z,z^*,S(y^*)))> \epsilon.
	\end{equation}

	The proof is by a critical scaling argument. We can choose $r_i\to\infty$ and $s_i\to\infty$ satisfying that 
	\begin{enumerate}
		\item $r_is_i^{-1}\to\infty$,
		\item $(r_i^{-1}\tilde M,\tilde p,\Gamma)\GH(Y,y^*,H),$
		\item $(s_i^{-1}\tilde M,\tilde p,\Gamma)\GH (X,x^*,G)$.
	\end{enumerate}
	
	Put $(N_i,p_i,\Gamma):=(r_i^{-1}\tilde M,\tilde p,\Gamma)$ and $\Lambda_i:=r_is_i^{-1}$. Then the above properties become,
	\begin{enumerate}
		\item $\Lambda_i\to\infty$,
		\item $(N_i,p_i,\Gamma)\GH(Y,y^*,H)$,
		\item $(\Lambda_iN_i,p_i,\Gamma)\GH(X,x^*,G)$.
	\end{enumerate}
	
	For each $i$, define 
	\begin{align*}
		A_i:=&\{R\in[1,\Lambda_i]|d_{GH}((RN_i,p_i,\Gamma(p_i)),(Z,z^*,S(y^*)))\le 0.001\epsilon,\\&\text{for some }(Z,z^*,S)\in \Omega(\tilde M,\Gamma)\text{ satisfying } k(Z)=2\text{ and (\ref{itm:a1})}\}.
	\end{align*}

	Obvious for all large $i$, $\Lambda_i\in A_i$. By contradiction assumption (\ref{0:31}), for every large $i$, $1\notin A_i$. Choose $\lambda_i\in A_i$ such that $\inf A_i\le \lambda_i<\inf A_i+i^{-1}$. 
	
	We claim that $\lambda_i\to\infty$. Suppose not, passing to a subsequence, we may assume $\lambda_i\to C\in[1,+\infty)$. Then
	$$(\lambda_iN_i,p_i,\Gamma)\GH(CY,y^*,H).$$
	By the choice of $\lambda_i$ and $A_i$, the above convergence yields $$d_{GH}((CY,y^*,H(y^*)),(Z,z^*,S(z^*)))\le0.002\epsilon,$$for some $(Z,z^*,S)\in \Omega(\tilde M,\Gamma)$ satisfying $k(Z)=2$ and (\ref{itm:a1}). Note that $S(z^*)$ is isometric to $\R^2$. By Lemma \ref{T2T3}, the space $(CY, y^*, H)$ satisfies that $CY$ is isometric to $\R^2\times C(S_r^1)$ and either condition (\ref{itm:a1}) or (\ref{itm:b1}) for $a \in (0, 0.006\epsilon)$. In either case, the following inequality holds:
	$$
	d_{GH}((Y, y^*, H(y^*)), (Z, z^*, S(z^*))) \le 0.002\epsilon,
	$$which contradicts the contradiction assumption (\ref{0:31}). So the claim follows.
	
	Up to a subsequence, we assume $$(\lambda_iN_i,p_i,\Gamma)\GH(Y_1,y_1^*,H_1).$$
	By definition of $\lambda_i$ and $A_i$, there exists $(Z,z^*,S)\in\Omega(\tilde M,\Gamma)$ satisfying $k(Z)=2$ and (\ref{itm:a1}) such that $$d_{GH}((Y_1,y_1^*,H_1(y_1^*)),(Z,z^*,S(z^*)))\le0.002\epsilon.$$
	Similar to the above, by Lemma \ref{T2T3}, the space $(Y_1, y_1^*, H_1)$ satisfies that $Y_1$ is isometric to $\R^2\times C(S_r^1)$ and either condition (\ref{itm:a1}) or (\ref{itm:b1}) for $a \in (0, 0.006\epsilon)$. In either case, the following inequality holds:
	$$
	d_{GH}((0.1Y_1, y_1^*, H_1(y_1^*)), (Z, z^*, S(z^*))) \le 0.0004\epsilon.
	$$Combing the fact that $\lambda_i\to\infty$, the above convergence implies that for large $i$, $0.1\lambda_i\in A_i$. So $\inf A_i\le 0.1\lambda_i<0.1\inf A_i+0.1\epsilon_i$ where $\epsilon_i$ is a subsequence of $i^{-1}$. This concludes that $1\le \inf A_i<\epsilon_i\to0$, which is impossible.
	
	So the proof of the sublemma is complete.

\end{proof}

\end{proof}

\begin{cor}\label{(b1)}
Every $(X,x^*,G)\in\Omega(\tilde M,\Gamma)$ does not satisfy (\ref{itm:b1}).
\end{cor}

\begin{proof}
Suppose there is some $(X,x^*,G)\in\Omega(\tilde M,\Gamma)$ satisfying (\ref{itm:b1}). Then by a blow down, there exists some $(Y,y^*,H)\in\Omega(\tilde M,\Gamma)$ satisfying (\ref{itm:a1}). According to Lemma \ref{(a1)}, every element of $\Omega(\tilde M,\Gamma)$ satisfies (\ref{itm:a1}). This is a contradiction.
\end{proof}

\begin{lem}\label{DimensionofIsotropyGroup}
	If there exists $(X,x^*,G)\in\Omega(\tilde M,\Gamma)$ not satisfying (\ref{itm:a1}), then
\begin{align}\label{5.10}
\max\{\dim(I(H))|(Y,y^*,H)\in \Omega(\tilde M,\Gamma)\}\le 1.
\end{align}
\end{lem}
\begin{proof}

By Lemma \ref{(a1)} and Corollary \ref{(b1)}, every $(Y,y^*,H)\in\Omega(\tilde M,\Gamma)$ does not satisfy (\ref{itm:a1}) and (\ref{itm:b1}). The proof is by examining all remaining possibilities listed in Lemma \ref{T2T3}.
	
	If $k(Y)=2$, then $Y$ is isometric to $\R^2\times C(S_r^1)$ with $r\in(0,1)$. In either case of (\ref{itm:a0}) or (\ref{itm:b2}), $I(H)$ is a closed subgroup of $\Isom (S_r^1)$, while in case of (\ref{itm:a2}), $I(H)$ is subgroup of the semidirect product of $\Z_2$ and $\Isom(S_r^1)$. Hence we have $\dim(I(H))\leq 1$.

If $k(Y)=1$, then $Y$ is isometric to $\R\times C(Z)$ for some non-collapsed $\RCD(1,2)$-space $Z$ with $\diam(Z)<\pi$ and $I(H)$ is a closed subgroup of $\Isom (Z)$.
By combining the principal orbit theorem for $\RCD$ spaces (see \cite[Theorem 4.7]{GGKMS18}) with the fact that $H$ is abelian, the principal orbit of the $I(H)$-action is homeomorphic to $I(H)$. Note that by \cite[Theorem 8.1]{BPS24}, $Z$ is homeomorphic to the $2$-sphere $S^2$. Hence, no $2$-torus can be embedded in $Z$. In conclusion, we have $\dim(I(H))\le 1$.

\end{proof}

\begin{lem}\label{(b2)}
	Every $(X,x^*,G)\in\Omega(\tilde M,\Gamma)$ does not satisfy (\ref{itm:b2}).
\end{lem}
\begin{proof}
	Argue by contradiction. Assume that there exists an $(X,x^*,G)\in\Omega(\tilde M,\Gamma)$ satisfying (\ref{itm:b2}). Let $h\in\Isom (G)$ such that $P_1(h)=r_b$. Then there exists $r_i\to\infty$ such that, $$(r_i^{-1}X,x^*,h,G)\GH(X,x^*,h_\infty,H).$$
	
	It is obvious that $I(H)=\spa{I(G),h_\infty}$ with $h_\infty\notin I(G)$.
	
	On the other hand, according to Lemma \ref{DimensionofIsotropyGroup}, (\ref{5.10}) holds, and then by Lemma \ref{StabilityOfIsotropy}, $I(G)$ is isomorphic to $I(H)$. This yields a contradiction.
\end{proof}

Combining Lemmas \ref{T2T3}, \ref{(a1)}, \ref{(b2)}, and Corollary \ref{(b1)}, there are only two possibilities of $(\tilde M,\Gamma)$:
\customitemize{P}
	\item \label{itm:P1}  Every $(X,x^*,G)\in\Omega(\tilde M,\Gamma)$ belongs to (\ref{itm:a1}).
	\item \label{itm:P2} Every $(X, x^*, G) \in \Omega(\tilde{M}, \Gamma)$ satisfies $k(X) = 1$ or belongs to either (\ref{itm:a0}) or (\ref{itm:a2}).
\end{enumerate}
Obviously, in the case (\ref{itm:P1}), $G(x^*)$ is isometric to $\R^2$, while in the case (\ref{itm:P2}), $G(x^*)$ is always isometric to $\R^1$. Note that by the connectedness of $\Omega(\tilde M,\Gamma)$ (see, for example, \cite[Lemma A.1]{H24}), (\ref{itm:P1}) and (\ref{itm:P2}) cannot coexist, The proof of Theorem \ref{OrbitisConnected} is complete.

\section{Proof of Theorem \ref{thm-Almost Abelian'} and Corollary \ref{cor-Polarness}}\label{section6}

\subsection{Proof of Theorem \ref{thm-Almost Abelian'} (1)}\label{section6.1}

By Theorem \ref{thm-finite generation}, $\Gamma:=\pi_1(M)$ is finitely generated.

We only need to prove that $\Gamma$ is a crystallographic group of rank $k$ ($1\leq k\leq3$). Once this is proved, then according to the Bieberbach theorems (\cite{Bieb11-1}), $\Gamma$ contains a normal subgroup isomorphic to $\mathbb{Z}^{k}$ of index $C$, where $C$ is a universal constant.
We divide the proof into two steps.

\textbf{Step 1:} we first assume that $\Gamma$ is nilpotent, and we will prove that $\Gamma$ is a finitely generated torsion-free abelian group. That is, $\Gamma$ is isomorphic to $\Z^k$. By \cite{An90-II}, $k\in\{1,2,3\}$.

\begin{claim}\label{claim torsion-free}
	$\Gamma$ is torsion-free.
\end{claim}

\begin{proof}[Proof of the Claim \ref{claim torsion-free}: ]
	Since $\Gamma$ is a finitely generated nilpotent group, the torsion subgroup $T$ of $\Gamma$ is a finite normal subgroup, and $\Gamma/T$ is infinite.
	$N:=\tilde{M}/T$ has Euclidean volume growth since $\tilde{M}$ has Euclidean volume growth.
By the assumption on $M$, it is easy to see that every asymptotic cone of $N$ splits an $\R^{n-4}$-factor, and the limit actions of $\Gamma/T$ on the asymptotic cones induce a trivial action on the $\R^{n-4}$-factor.
	Applying Lemma \ref{Lem17:28} (1) to the group action $\Gamma/T$ on $N$, we conclude that any asymptotic cone of $N$ splits an $\R^{n-3}$-factor.
	Then by Theorem \ref{SimplyConnectnessOf(n-3)symmetry}, $N$ is simply connected, and we conclude that $T$ is trivial.
\end{proof}

Before proceeding to the next claim, we first prove a lemma.

\begin{lem}\label{StableSplittingLemma}
	Let $N$ be an open $n$-manifold with nonnegative Ricci curvature with nilpotent $\pi_1(N)$. If the universal cover $\tilde N$ has Euclidean volume growth, and there exists $k$, such that, every asymptotic cone of $N$ splits an $\R^k$-factor and does not split an $\R^{k+1}$-factor, then $\pi_1(N)$ is finitely generated and virtually abelian.
\end{lem}
\begin{proof}
	
	The finite generation of $\pi_1(N)$ is implied by Lemma \ref{splitting-lem}. For the virtual abelianness of $\pi_1(N)$, we need the following sublemma.
	\begin{slem}\label{slem5.6}
		Under the assumption of Lemma \ref{StableSplittingLemma}, there exists an integer $s\in[k,n]$ such that for any $(Y,y,G)\in\Omega(\tilde N,\pi_1(N))$, $(Y,y,G(y))$ is isometric to $(\R^s\times X,(0^s,x),\R^s\times\{x\})$ for some $X$.
	\end{slem}
	\begin{proof}

Under the assumption of Lemma \ref{StableSplittingLemma}, every asymptotic cone of $\tilde{N}$ splits an $\R^k$-factor.
For any $r_i\to\infty$, up to a subsequence, we may assume the following diagram holds,
		
		\begin{equation*}
			\xymatrix{
				(r_i^{-1}\tilde N,\tilde p,\pi_1(N)) \ar[rr]^{GH}\ar[d]_{}&&(\R^k\times \R^l\times \bar X,(0^k,0^l,\bar x),G)\ar[d]^{} \\
				(r_i^{-1}N,p)\ar[rr]^{GH}&  & (\R^k\times \underline{X},(0^k,\underline{x})),}
		\end{equation*}
		where both $\bar X$ and $\underline{X}$ are metric cones which splits no $\R$-factor, and $G$ acts on the $\R^k$-factor trivially. Since every asymptotic cone of $N$ does not split $\R^{k+1}$-factor, every tangent cone of $\underline X$ at $\underline x$ splits no $\R$-factor. Since $\pi_1(N)$ is nilpotent, so is $G$. Now applying \cite[Lemma 3.1]{H24} to $(\R^l\times\overline{X},(0^l,\overline x),G)$ yields that $(Y,y,G(y))$ is isometric to $(\R^s\times X,(0^s,x),\R^s\times\{x\})$ for some $s$ and $X$. The fact that $s$ does not depend on the choice of $(Y,y,G)$ is implied by the connectedness of $\Omega(\tilde N,\pi_1(N))$ (for example, see \cite[Lemma A.1]{H24}).
	\end{proof}
The virtual abelianness of $\pi_1(N)$ follows by Sublemma \ref{slem5.6} and \cite[Theorem A, Proposition 4.2]{Pan22}.
\end{proof}

\begin{claim}\label{claim3.3}
	$\Gamma$ is virtually abelian.
\end{claim}
\begin{proof}
	By the assumptions on $M$ and $\Gamma$, if every asymptotic cone of $M$ does not split $\R^{n-3}$-factor, then according to Lemma \ref{StableSplittingLemma}, $\Gamma$ is virtually abelian.
	
	Hence in the following we assume that, there exists a sequence of $r_{i}\rightarrow\infty$ such that $(r_{i}^{-1}M,p) \GH (\R^{n-3}\times Y,(0,y^*))$.
	It is not hard to see that
	\begin{itemize}
		\item [(R)]\label{Splits}	Any normal covering $(\bar{M},\bar p ,\bar\Gamma)$ of $(M,p)$, up to a subsequence of $i$, the equivariant Gromov-Hausdorff limit of $(r_{i}^{-1}\bar{M},\bar p,\bar\Gamma)$ splits an $\R^{n-3}$-factor, on which the limit group acts trivially.

	\end{itemize}
	
	Since $\Gamma$ is nilpotent, the center $H$ of $\Gamma$ is not trivial.
	By Claim \ref{claim torsion-free}, $H$ is an infinite, finitely generated, torsion-free abelian group.

	It suffices to show $\Gamma/H$ is finite. Considering the $\Gamma/H$-action on $\hat{M}:=\tilde{M}/H$, by (R), after passing to a subsequence, we have the following commutative diagram of equivariant Gromov-Hausdorff convergence,

	\begin{align*}
		\xymatrix@C=2.5cm{
			(r_{i}^{-1}\tilde{M}, \tilde{p}, H) \ar[d]_{} \ar[r]^{GH} & (\R^{n-3}\times X, (0^{n-3},x^{*}), H_{\infty}) \ar[d]^{} \\
			(r_{i}^{-1}\hat{M}, \hat{p}, \Gamma/H) \ar[r]^{GH} & (\R^{n-3}\times (X/H_{\infty}),(0^{n-3},\hat x),K_\infty),  }
	\end{align*}
	where $X$ is a $3$-dimensional metric cone with vertex $x^{*}$.
	By (R), $H_\infty$ and $K_\infty$ only act on the $X$-factor and $X/H_\infty$-factor respectively.
	According to Lemma \ref{Lem17:28}, we conclude that $X$ splits an $\R$-factor. Consequently, $X$ is isometric to $\R\times C(S^{1}_{r})$, where $S^{1}_{r}$ is a circle with radius $r\in(0, 1]$.
	The case of $r=1$ will implies that $\R^{n-3}\times X$ is isometric to $\R^{n}$, then by Colding's rigidity theorem on maximal volume growth (\cite{Col97}), $\tilde{M}$ is isometric to $\R^{n}$, and in this case the conclusion is classical.
	Hence in the following we assume $r\in(0, 1)$. By Lemma \ref{Lem17:28}, $H_{\infty}$ acts transitively on the $\R$-factor of $X$.
	Thus $(X/H_{\infty},\hat x)$ is a metric cone with the vertex $\hat x$, which contains no line.
	So $K_\infty(\hat x) = \{\hat x\}$. Since $\Gamma/H$ is finitely generated, according to Lemma \ref{Lem17:28} (1), $\Gamma/H$ is finite.
	This completes the proof of Claim \ref{claim3.3}.
\end{proof}

Combining Claims \ref{claim torsion-free} and \ref{claim3.3} with the following algebraic lemma, we conclude that $\Gamma$ is free abelian.
This completes Step 1.

\begin{lem}\label{lem3.5}
	Let $G$ be a virtually abelian torsion-free nilpotent group, then $G$ is abelian.
\end{lem}

Lemma \ref{lem3.5} is known for experts for group theory.
We include its proof in appendix for readers' convenience.

\textbf{Step 2:} We consider the general case that $\Gamma:=\pi_{1}(M)$ is infinite.
Since $\Gamma$ is finitely generated, according to Corollary 4 of \cite{KW11}, $\Gamma$ contains a nilpotent normal subgroup, denoted by $H$, of finite index.
Let $\hat{M}=\tilde{M}/H$, then by Step 1, we conclude that $H$ is isomorphic to $\mathbb{Z}^{k}$ ($1\leq k\leq3$).
Thus $\Gamma$ is virtually $\mathbb{Z}^{k}$.
By a characterization of virtually $\mathbb{Z}^{k}$ groups (see \cite[Theorem 2.1]{Wil00}), there exist a finite group $E$ and a crystallographic group $G$ such that the following short exact sequence holds:
\begin{align*}
	\{e\}\rightarrow E \rightarrow \Gamma\rightarrow G\rightarrow\{e\}.
\end{align*}
Then $\bar{M}=\tilde{M}/E$, which has Euclidean volume growth, is a covering of $M$ with deck transformation $\Gamma/E\cong G$.
By the assumption on $M$ and Lemma \ref{Lem17:28}, every asymptotic cone of $\bar M$ splits an $\R^{n-3}$-factor.
Then by Theorem \ref{SimplyConnectnessOf(n-3)symmetry}, $\bar{M}$ is simply connected.
Thus $E$ is trivial, and hence $\Gamma$ is isomorphic to the crystallographic group $G$.
The proof is complete.

\subsection{Proof of Theorem \ref{thm-Almost Abelian'} (2)}\label{section6.2}

We need the following lemma of stability of finite group actions.

\begin{lem}\label{StabilityOfFiniteGroupsAction}
	Let $(M_i,p_i)$ be a sequence of open $n$-manifolds with $\Ric_{M_i}\ge-(n-1)$ and $\vol(B_1(p_i))>v>0$, and $\pi_i:(\check M_i,\check p_i)\to(M_i,p_i)$ be a normal covering with deck transformation $\Gamma_i$. Suppose that the following commutative diagram of equivariant Gromov-Hausdorff convergence holds,
	
	\begin{equation*}
		\xymatrix@C=2.5cm{
			(\check{M}_i, \check{p}_i,\Gamma_i) \ar[d]_{\pi_i} \ar[r]^{GH} &(Y,y_0,G) \ar[d]^{\pi} \\
			(M_i, p_i) \ar[r]^{GH} &
			(X,x_0).}
	\end{equation*}
	Let the triple of maps $f_i:\left(B_{\epsilon_i^{-1}}(\check p_i),\check p_i\right)\to\left( B_{\epsilon_i^{-1}+\epsilon_i}(y_0),y_0\right)$, $\phi_i:\Gamma_i\to G$, $\psi_i:G\to\Gamma_i$ be an $\epsilon_i$-equivariant pointed Gromov-Hausdorff approximation, where $\epsilon_i\to0$. If \begin{equation}\label{1:30}
		\sup\diam(\Gamma_i(\check p_i))<\infty,
	\end{equation}
	then for all large $i$, $\phi_i,\psi_i$ are both isomorphisms with $\phi_i=\psi_i^{-1}$.
\end{lem}

Lemma \ref{StabilityOfFiniteGroupsAction} is well-known among experts. We have included its proof in the appendix for the reader's convenience.

We begin the proof of Theorem \ref{thm-Almost Abelian'}. By the assumptions and Cheeger-Colding's theory, there exists a sequence of $r_{i}\to\infty$ such that the following commutative diagram of equivariant Gromov-Hausdorff convergence holds,
\begin{align}\label{diag2.5}
	\xymatrix@C=2.5cm{
		(r_{i}^{-1}\tilde{M}, \tilde{p}, \Gamma) \ar[d]_{} \ar[r]^{GH} & (\R^{n-4}\times C(X), (0^{n-4},x^{*}), G) \ar[d]^{} \\
		(r_{i}^{-1}M, p) \ar[r]^{GH} & (\R^{n-4}\times C(Y),(0^{n-4},y^{*})),  }
\end{align}
where $\Gamma:=\pi_1(M)$, and, $C(X)$ and $C(Y)$ are metric cones with vertex $x^{*}$ and $y^{*}$ respectively and $G$ acts on the $\R^{n-4}$-factor trivially.

Since $\Gamma$ is finite, it is easy to see $G(x^{*})=\{x^{*}\}$, and hence $G$ induces an isometric action on $X$, with $X/G=Y$.

For simplicity of notions, put $(M_i,p_i):=(r_i^{-1}M,p)$ and $(\tilde M_i,\tilde p_i):=(r_i^{-1}\tilde M,\tilde p)$.

By Lemma \ref{StabilityOfFiniteGroupsAction}, there exist $\epsilon_i\to0$, and a triple of maps $f_{i}:B_{\epsilon_i^{-1}}(\tilde p_i)\to B_{\epsilon_i^{-1}+\epsilon_i}(0^{n-4},x^*)$, $\psi_i^{-1}:\Gamma\to G$, $\psi_i:G\to\Gamma$ which is an $\epsilon_i$-equivariant Gromov-Hausdorff approximation, such that $\psi_i$ is an isomorphism. Since $\Gamma$ and $G$ are both finite, up to a subsequence, we may assume $\psi_i\equiv\psi:G\to\Gamma$ for every $i$.

\begin{claim}\label{ClaimFreeAction}
	The $G$-action on $X$ is free. In particular, the quotient $\pi:X\rightarrow Y=X/G$ is a normal cover.
\end{claim}
\begin{proof}
	Suppose on the contrary, there exist some $x_{0}\in X$ and a non-trivial subgroup $H$ of $G$ such that $H(x_{0})=\{x_{0}\}$. Let $q:=(0^{n-4},1,x_0)\in\R^{n-4}\times C(X)$. Then there exists a sequence of $\tilde x_i\in \tilde M_i$ such that $(\tilde M_i,\tilde x_i,\psi(H))\GH(\R^{n-4}\times C(X),q,H)$, and $\diam^{\tilde M_i} (\psi(H)(\tilde x_i))\to0$.

By the cone splitting principle, every tangent cone of $\R^{n-4}\times C(X)$ at $q$ is isometric to $(\R^{n-3}\times C(Z),(0^{n-3},z^*))$ for some $Z$.
By the standard diagonal argument, up to a subsequence, we may assume that there exist $s_i\to\infty$ satisfying
	\customitemize{p}
		
		\item \label{itm:p1}
		$(s_i(\R^{n-4}\times C(X)),q,H)\GH(\R^{n-3}\times C(Z),(0^{n-3},z^*),K)$,
		\item \label{itm:p2}
		$(s_i\tilde M_i,\tilde x_i,\psi(H))\GH(\R^{n-3}\times C(Z),(0^{n-3},z^*),K)$,
		\item \label{itm:p3}
		$s_ir_i^{-1}\to0$ and $s_i\diam^{\tilde M_i} (\psi(H)(\tilde x_i))\to0$.
		
	\end{enumerate}
	
	Note that the $H$-action on $\R^{n-4}\times C(X)$ fixes all point of the form $(a,r,x_{0})\in C(X)$, where $a\in\R^{n-4}$, $r\ge0$. Hence by (p1), $K$ acts on $\R^{n-3}\times \{z^*\}$ trivially. Consequently $K$ induces an isometric action on $Z$. Hence by (\ref{itm:p2}), $$(s_i\tilde M_i/\psi(H),[\tilde x_i])\GH(\R^{n-3}\times C(Z/K),(0^{n-3},[z^*])).$$
	
	It is easy to verify that there exists $v > 0$ such that $\vol(B_1([\tilde{x}_i])) > v$. Combining this with (\ref{itm:p3}), we are now in a position to apply Lemma \ref{KeyLemma} to the above convergent sequence, which implies that $\psi(H)$ is trivial. Consequently, $H$ is trivial, which leads to a contradiction. 	
\end{proof}

We continue to study the commutative diagram (\ref{diag2.5}).
According to \cite[Theorem 1.4 ]{BPS24}, $X$ (and $Y$ respectively) is homeomorphic to some $3$-dimensional space form $S^{3}/\Gamma_{X}$ (and $S^{3}/\Gamma_{Y}$ respectively). By Claim \ref{ClaimFreeAction}, $X$ is a normal cover of $Y$ with deck transformation $G=\Gamma$. This concludes that $\Gamma_X$ is a normal subgroup of $\Gamma_Y$, and $\Gamma$ is isomorphic to $\Gamma_Y/\Gamma_X$.
According to \cite{Wolf11}, the fundamental group of a spherical $n$-manifold always contains a normal cyclic subgroup of index at most $C(n)$ (see \cite{Rong96}, \cite{Rong24} for some extentions).
Thus $\Gamma$ contains a normal cyclic subgroup of index at most $C=C(3)$.
The proof is complete.

\subsection{Proof of Corollary \ref{cor-Polarness}}\label{section6.3}

Without loss of generality, we assume that $M$ is not flat.

If $\pi_1(M)$ is finite, then $M$ has Euclidean volume growth. By Cheeger-Colding's theory, any asymptotic cone of $M$ is a metric cone with the reference point as a cone vertex. So the first conclusion holds. And the second conclusion holds trivially.

If $\pi_1(M)$ is infinite, then by Theorem \ref{thm-Almost Abelian}, $\pi_1(M)$ is virtually abelian. Let $\Gamma$ be the abelian subgroup of $\pi_1(M)$ with finite index.
Consider the following diagram for $r_i\to\infty$:
\begin{equation*}
	\xymatrix@C=2.5cm{
		(r_i^{-1}\tilde{M}, \tilde{p},\Gamma,\pi_1(M)) \ar[d]_{} \ar[r]^{GH} &(X,x^*,H,G) \ar[d]^{} \\
		(r_i^{-1} M,  p) \ar[r]^{GH} &
		(Y,y^*).}
\end{equation*}

Applying Theorem \ref{OrbitisConnected} to $\Omega(\tilde M,\Gamma)$, we conclude that $H(x^*)$ is isometric to $\R^l$, where $l\in\{1,2\}$ does not depend on the choice of $r_i$. Since $\pi_1(M)/\Gamma$ is finite, it follows that $G(x^*)=H(x^*)$. So $G(x^*)$ is isometric to $\R^l$, where $l$ is described as above. This implies that $Y=X/G$ has a pole at $y^*$, which proves the first conclusion. And the second conclusion follows by applying \cite[Proposition 4.2]{Pan22}.

\section{Appendix}\label{section7}

\subsection{Proof of Lemma \ref{Lem17:28}}

\begin{proof}[Proof of Lemma \ref{Lem17:28} (1)]
	We assume that
	\begin{equation}\label{0:59}
		(r_i^{-1}M,p,\Gamma)\GH(Y,y^*,G),
	\end{equation}
	for some $r_i\to\infty$. By Colding-Naber \cite{CN12}, $G$ is a Lie group.
	
We claim that for each $r>0$, there exists $g\in G$, such that $g(y^*)\in \partial B_r(y^*)$. This obviously implies the required conclusion. Let $\{\gamma_1,\gamma_2,\ldots,\gamma_m\}\subset\Gamma$ be a generator of $\Gamma$.
Since $\Gamma(p)$ is unbounded, for each $i$, there exists $\omega_i\in\Gamma$ such that $d_i(\omega_i(p),p)>r$, where $d_i$ is the distance on $r_i^{-1}M$. Assume $\omega_i=\gamma_{j_1}\gamma_{j_2}...\gamma_{j_{k_i}}$. For each $s=1,2,\ldots,k_{i}$, put $x_{i,s}:=\gamma_{j_1}\gamma_{j_2}...\gamma_{j_{s}}(p)$ and $x_{i,0}:=p$. Then $$d_i(x_{i,0},p)<r<d_i(x_{i,k_i},p),\,\quad d_i(x_{i,s},x_{i,s+1})=r_i^{-1}d(p,\gamma_{j_{s+1}}(p))\le r_i^{-1}D,$$where $D:=\max\{d(\gamma_j(p),p)|j=1,2,\ldots,m\}$. Choose $s_i$ such that $$d_i(x_{i,s_i},p)\le r<d_i(x_{i,s_i+1},p).$$
Then up to a subsequence, we assume that $\omega_i':=\gamma_{j_1}\gamma_{j_2}...\gamma_{j_{s_i}}$ and $\omega_i'':=\gamma_{j_1}\gamma_{j_2}...\gamma_{j_{s_i+1}}$ converges to $\omega_0,\omega_1\in G$ respect to (\ref{0:59}), respectively. Hence $$d(\omega_0(y^*),y^*)\le r\le d(\omega_1(y^*),y^*),\,\quad d(\omega_0(y^*),\omega_1(y^*))=0$$
	which implies $\omega_0(y^*)\in\partial B_r(y^*)$. The claim follows.	
	
\end{proof}

\begin{proof}[Proof of Lemma \ref{Lem17:28} (2)]
We assume
\begin{equation*}
		(r_i^{-1}M,p,\spa{\gamma})\GH(Y,y^*,H),
	\end{equation*}
	for some $r_i\to\infty$. Then $H$ is an abelian closed subgroup of $\Isom(Y)$. Again by Colding-Naber \cite{CN12}, $H$ is a Lie group.
	
	By the assumptions of $\gamma$, $\spa{\gamma}(p)$ is unbounded. By Lemma \ref{Lem17:28} (1), $H(y^*)$ is unbounded and not discrete. So the identity component $H_0$ of $H$ has a positive dimension. Assume that $H_0$ is isomorphic to $\R^l\times T^s$. If we can show that $H_0(y^*)$ is unbounded, then it follows that $l \ge 1$. Consequently, there exists a closed subgroup isomorphic to $\R^l$, denoted by $L$, of $H_0$.
Since $L$ does not contain any nontrivial compact subgroups, we have $G_{x^*}\cap L=\emptyset$, where $G_{x^*}$ is the isotropy of $x^*$.
So $L(y^*)$ is homemorphic to $\R^l$.
	
	Now we verify that $H_0(y^*)$ is unbounded. It suffices to verify that for any $r>1$, $\epsilon\in(0,0.1)$, there exist $x_0:= y^*,x_1,\ldots,x_k\in H(y^*)$, such that for $s=0,1,\ldots,k-1$, $d(x_s,x_{s+1})\le\epsilon$ and $d(x_k,y^*)\ge r$.

Fix a large $i$ such that $r_i^{-1}d(p,\gamma(p))+d_{GH}((r_i^{-1}M,p,\spa{\gamma}),(Y,y^*,H))\le0.01\epsilon$. Let $f_i:(r_i^{-1}M,p)\to(Y,y^*)$ and $\phi_i:\spa{\gamma}\to H$ be a pair of maps which form a $0.01\epsilon$-equivariant Gromov-Hausdorff approximation.
Let $k$ be the minimal integer such that $d_i(\gamma^{k}(p),p)\ge r+\epsilon$.
Then, $d_i(\gamma^{s}(p),\gamma^{s+1}(p))\le 0.2\epsilon$ and $d_{i}(\gamma^{s+1}(p_i),p_i)\le r+2\epsilon$, for $s=0,1,\ldots,k-1$.
Define $x_s:=\phi_i(\gamma^s)(y^*)$, $i=0,1,\ldots,k$. Then it is direct to verify such $x_s$ satisfy the required property.
\end{proof}
	
	If $M$ has Euclidean volume growth, by Cheeger-Colding's theory (\cite{CC96}), $(Y,y^*)$ is isometric to a metric cone $(C(X),y^*)$ with a vertex $y^{*}$.
	Since $G(y^*)$ or $H(y^*)$ is not a single point, we conclude that $Y$ splits an $\R$-factor.

\subsection{Proof of Lemma \ref{T2T3}}

The case that $k(X)=1$ is obvious. And when $k(X)=2$, the structure of $(X,x^*)$ is a simple consequence of Cheeger-Colding's theory. So we just present the proof for the structure of $K$ in the case that $k(X)=2$.

\begin{lem}\label{Lem:Pertainaline}
	Let $f:\R^n\to\R^n$ be a translation defined by $f(v)=v+v_0$, $v_0\in\R^n\setminus \{0^{n}\}$. If $g\in\Isom(\R^n)$ and $f\circ g=g\circ f$, then $g(v+tv_0)=g(v)+tv_0$ for any $v\in\R^n$ and $t\in(-\infty,+\infty)$.
\end{lem}
\begin{proof}
	Let $g(v)=Av+w_0$, where $(A,w_0)\in O(n)\ltimes\R^n$. So by $f\circ g=g\circ f$, $Av_0=v_0$. Therefore, $g(v+tv_0)=Av+tAv_0+w_0=(Av+w_0)+tv_0=g(v)+tv_0$.
\end{proof}

\begin{lem}\label{R-actionOnR^2}
	Let $L$ be a closed subgroup of $\Isom(\R^2)$ which is isomorphic to $\R$. Then $L$ acts on $\R^2$ by translation.
\end{lem}
\begin{proof}

Let $\mathrm{Pj}: \Isom(\R^2)\rightarrow O(2)$ be the standard homomorphism given by $\mathrm{Pj}(g)=A$, for $g(v)=Av+v_0$, $\forall v\in \R^2$. Since $\mathrm{Pj}$ is a Lie group homomorphism, it concludes that $\mathrm{Pj}(L)$ is a connected subgroup of $O(2)$. Hence $\mathrm{Pj}(L)$ is trivial or contains all rotations.

Note that since $L$ is torsion-free, the $L$-action is free. So for any non-trivial $g\in L$ with $g(v)=Av+v_0$, the equation $g(v)=v$, i.e., $(I-A)v=v_0$, with respect to $v$, has no solution. In particular, $A$ has an eigenvalue $1$. Based on the discussion in the previous paragraph, the determinant of $A$ is positive. So the other eigenvalue of $A$ is also $1$. This concludes $A=I$ and $\mathrm{Pj}(L)=\{I\}$.

\end{proof}

The structure of $K$ as in Lemma \ref{T2T3} is implied by the following lemma.
\begin{lem}
	Let $K$ be a closed abelian subgroup of $\Isom (\R^2)$. If $K$ contains a closed $\R$-subgroup, then up to a re-choice of the normal coordinate of $\R^2$, $K$ satisfies one of the following.
	\begin{enumerate}
		\item If $K(0^2)$ is connected, then one of the following holds.
		\customitemize{a}

			\item \label{itm:a1'} $K=\spa{t_1,s_2|t,s\in(-\infty,+\infty)}$.
			
			\item \label{itm:a2'} $K=\spa{t_1,r_0|t\in(-\infty,+\infty)}$.
			
			\item \label{itm:a0'} $K=\spa{t_1|t\in(-\infty,+\infty)}$.
		\end{enumerate}
		
		\item  If $K(0^2)$ is not connected, then one of the following holds.
		\customitemize{b}
			\item \label{itm:b1'} There exists $a>0 $ such that $K=\spa{t_1,a_2|t\in(-\infty,+\infty)}$.
			
			\item \label{itm:b2'} There exists $b>0 $ such that $K=\spa{t_1,r_b|t\in(-\infty,+\infty)}$.
		\end{enumerate}
	\end{enumerate}
	
	Where $t_1,s_2,a_2,r_b$ are defined as in Lemma \ref{T2T3}.
	
\end{lem}
\begin{proof}
	
	Since $K$ contains a closed $\R$-subgroup, by Lemma \ref{R-actionOnR^2}, this $\R$-subgroup acts on $\R^2$ by translation. So we may re-choose the normal coordinate on $\R^2$ such that this $\R$-action is determined by the $\spa{t_1|t\in(-\infty,+\infty)}$-action. If $K=\spa{t_1|t\in(-\infty,+\infty)}$, this is the case of (\ref{itm:a0'}). Hence we assume $K\backslash\spa{t_1|t\in(-\infty,+\infty)}\neq\emptyset$ below.

	Define $K_1:=\{g\in K|g\neq\id,\,g(0,0)=(0,b)\text{ for some }b\}$. Note that $K$ is generated by $\{t_1|t\in(-\infty,+\infty)\}$ and $K_1$; for any $k\in K$ with $k(0,0)=(a,b)$, $(-a)_1\circ k\in K_1$. So we just need to classify elements of $K_1$.
	
	Fix $g\in K_1$ satisfying $g(0,0)=(0,b)$. For any fixed $y$, let
	\begin{equation}\label{x_1y_1}
		g(0,y)=(x_1,y_1).
	\end{equation}
	Then for any $x$,
\begin{align}\label{7.5}
|(x,y)|=|(0,y)-(x,0)|=|g(0,y)-g(x,0)|=|(x_1,y_1)-(x,b)|=|(x_1-x,y_1-b)|,
\end{align}
where we have used Lemma \ref{Lem:Pertainaline} in the third equality.
By the arbitrariness of $x$ in (\ref{7.5}), it concludes that
	\begin{equation}\label{13:08}
		x_1=0,\,|y|=|y_1-b|.
	\end{equation}

	\begin{claim}\label{b=0}
		If $b=0$, then $g(x,y)=(x,-y)$.
	\end{claim}
	\begin{proof}
Choosing $y=1$ in (\ref{x_1y_1}), by (\ref{13:08}), we have either $g(0,1)=(0,1)$ or $g(0,1)=(0,-1)$.
Note that
$$g(0,0)=(0,0),\,g(1,0)=g(0,0)+(1,0)=(1,0),$$
where the second equality is due to Lemma \ref{Lem:Pertainaline}.
If $g(0,1)=(0,1)$, then $g$ is the identity, contradicting to $g\in K_{1}$. So $g(0,1)=(0,-1)$. The only possibility is that $g(x,y)=(x,-y)$.
	\end{proof}
	
	\begin{claim}\label{bnot0}
		If $b\neq 0$, then $g(x,y)=(x,y+b)$ or $g(x,y)=(x,b-y)$.
	\end{claim}
	\begin{proof}	
	Choosing $y=b$ in (\ref{x_1y_1}), by (\ref{13:08}), we conclude that one of the following cases holds,
	\customitemize{d}
		\item \label{itm:d1} $g(0,b)=(0,2b)$.
		\item \label{itm:d2} $g(0,b)=(0,0)$.
	\end{enumerate}

	Note that we also have $$g(0,0)=(0,b),\,g(1,0)=g(0,0)+(1,0)=(1,b).$$
	Hence the above two equalities and either of (\ref{itm:d1}) and (\ref{itm:d2}) determine a unique isometry. So $g(x,y)=(x,y+b)$ or $g(x,y)=(x,b-y)$.
	\end{proof}
	By combining Claim \ref{b=0} and Claim \ref{bnot0}, it follows that an element $g$ of $K_1$ can take the form of one of the following:
	\customitemize{f}
		\item \label{itm:f1} $g(x,y)=(x,-y)$.
		
		\item \label{itm:f2} $g(x,y)=(x,b-y)$, $b\neq 0$.
		
		\item \label{itm:f3} $g(x,y)=(x,y+b)$, $b\neq 0$.
	\end{enumerate}
	Due to the abelianness of $K$, $K_1$ cannot contain elements of any two of the above forms. Therefore, we divide the remaining discussion into three cases based on the form of the elements of $K_1$.
	
	If the elements of $K_1$ take the form (\ref{itm:f1}), then $K_1$ only contains a unique element $g(x,y)=(x,-y)$. This is the case of (\ref{itm:a2'}).
	
	If the elements of $K_1$ take the form (\ref{itm:f2}), then combining the abelianness of $K$, $K_1$ only contains a unique element $g(x,y)=(x,b-y)$, $b\neq 0$. This is the case of (\ref{itm:b2'}).
	
	In the remaining discussion, we assume that the elements of $K_1$ take the form (\ref{itm:f3}). Define $a:=\inf\{|b||(x,y+b)\in K_1\}$. If $a=0$, then by the assumption that $K$ is closed, for any $b\neq 0$, $(x,y+b)\in K_1$, which is the case of (\ref{itm:a1'}). If $a>0$, then again by the closedness of $K$, it follows that $(x, y+a) \in K_1$, which is the case of (\ref{itm:b1'}).
\end{proof}

\subsection{Proof of Lemma \ref{lem3.5}}

Let $H$ be a normal abelian subgroup of $G$ with finite index. For any $g_{1},g_{2}\in G$, there exist $m,n\in \mathbb{Z}^{+}$ such that $g_{1}^{m}\in H$, $g_{2}^{n}\in H$.
By the abelianness of $H$, we have $g_{1}^{m}g_{2}^{n}=g_{2}^{n}g_{1}^{m}$.
Then by \cite[16.2.9 Exercise]{KM79}, we have $g_{1}g_{2}=g_{2}g_{1}$.
Thus $G$ is abelian.

\subsection{Proof of Lemma \ref{StabilityOfFiniteGroupsAction}}

We argue by contradiction. By passing to a subsequence, we may assume that the conclusion fails for every $i$.

By (\ref{1:30}), we may assume $\diam(\Gamma_i(\check p_i))<D$. Then the cardinality of $\Gamma_i$,
\begin{equation}\label{BoundedOrder}
	\#\Gamma_i\le\frac{\vol(B_{1+D}(\check p_i))}{\vol (B_1(p_i))}\le \frac{C(n,D)}{v}.
\end{equation}
Observe that for all large $i$, $\#G\le\#\Gamma_i\le  \frac{C(n,D)}{v}$.

For any $y_i\in B_1(\check p_i)$ and arbitrary $\gamma_1,\gamma_2\in G$, by the definition of equivariant Gromov-Hausdorff approximation and the triangular inequality, we have
\begin{align*}
	&d(f_i(\psi_i(\gamma_1\gamma_2)(y_i)),f_i(\psi_i(\gamma_1)\psi_i(\gamma_2)(y_i)))\\\le& d(f_i(\psi_i(\gamma_1\gamma_2)(y_i)),\gamma_1\gamma_2f_{i}(y_i))+d(\gamma_1\gamma_2f_i(y_i),\gamma_1f_i(\psi(\gamma_2)(y_i)))\\&+d(\gamma_1f_i(\psi(\gamma_2)(y_i)),f_i(\psi_i(\gamma_1)\psi_i(\gamma_2)(y_i)))\\\le&3\epsilon_i,
\end{align*}
which implies
\begin{equation}\label{SmallDisplacement}
	d(y_i,\psi_i(\gamma_1\gamma_2)^{-1}\psi_i(\gamma_1)\psi_i(\gamma_2)(y_i))=d(\psi_i(\gamma_1\gamma_2)(y_i),\psi_i(\gamma_1)\psi_i(\gamma_2)(y_i))\le 4\epsilon_i.
\end{equation}

Let $H_i$ be the subgroup of $\Gamma_i$ generated by $\psi_i(\gamma_1\gamma_2)^{-1}\psi_i(\gamma_1)\psi_i(\gamma_2)$. Combining (\ref{BoundedOrder}) and (\ref{SmallDisplacement}), for any $y_i\in B_1(p_i)$, every orbit $H_i(y_i)$ has diameter $\le 4C(n,D)v^{-1}\epsilon_i$. Now by \cite[Theorem 0.8]{PR18}, for every sufficiently large $i$, $H_i$ is trivial. That implies $\psi_i$ is an homomorphism. Similarly, for large $i$, we can show that $\phi_i$ is an homomorphism. Further, by a similar estimate, we have for any $\gamma_i\in \Gamma_i$, $y_i\in B_1(\check p_i)$,
\begin{equation*}
	d(\gamma_i(y_i),\psi_i(\phi_i(\gamma_i))(y_i))\le 3\epsilon_i.
\end{equation*}
Again by \cite[Theorem 0.8]{PR18}, the above inequality and (\ref{BoundedOrder}) imply $\psi_i\circ\phi_i$ is the identity map of $\Gamma_i$. $\phi_i\circ\psi_i$ is similar. So for large $i$, $\psi_i$ and $\phi_i$ are both isomorphisms, which is a contradiction.

\vspace*{20pt}

\bibliographystyle{alpha}
\bibliography{ref}

\end{document}